\documentclass[11pt,leqno]{article}

\usepackage[active]{srcltx}

\usepackage{amsfonts,latexsym,amsmath,amssymb,amsthm}

\usepackage{fullpage}
\usepackage{dsfont}
\newtheorem{theorem}{Theorem}[section]
\newtheorem{lemma}[theorem]{Lemma}

\newtheorem{proposition}[theorem]{Proposition}
\newtheorem{corollary}[theorem]{Corollary}
\newtheorem{remark}[theorem]{Remark}

\theoremstyle{definition}

\theoremstyle{remark}

\newtheorem*{note*}{Note}

\numberwithin{equation}{section}

\makeatletter
\newcommand{\rank}{\mathop{\operator@font rank}}
\newcommand{\conv}{\mathop{\operator@font conv}}

\newcommand{\onetagright}{\tagsleft@false}
\makeatother

\newcommand{\ls}{\leqslant}
\newcommand{\gr}{\geqslant}
\renewcommand{\epsilon}{\varepsilon}

\newcommand{\E}{\mathbb{E}}

\newcommand{\R}{\mathbb{R}}

\newcommand{\e}{\varepsilon}

\newcommand{\vol}{\mathrm{vol}}

\def\e{\varepsilon}
\def\irr#1{{\rm Irr}(#1)}
\def\irrr#1#2 {\irr {#1 \mid #2}}

\usepackage{color}

\usepackage{hyperref}

\usepackage{setspace}
\begin{document}
\small

\title{\bf Functional perimeter and the dimensional Brunn-Minkowski inequality for log-concave measures}

\medskip

\author{Alexandros Eskenazis, Apostolos Giannopoulos and Natalia Tziotziou}

\date{}
\maketitle

\begin{abstract}\footnotesize
This paper is dedicated to two geometric problems associated to log-concave measures on $\R^n$. First, we study the dimensional Brunn--Minkowski inequality for even log-concave probability measures $\mu$ on $\mathbb{R}^n$ via an analytic approach based on diffusion operators and gradient estimates. We prove that for every pair of symmetric convex sets $K,L$ in $\R^n$ and every $\lambda\in(0,1)$,
$$\mu(\lambda K+(1-\lambda)L)^{c_n} \gr \lambda \mu(K)^{c_n}+(1-\lambda)\mu(L)^{c_n},$$
where $c_n\gr c/n^3\ln n$ for some absolute constant $c>0$. Secondly, we study the maximal perimeter $\Gamma(\mu)$ of an isotropic log-concave measure $\mu$, without symmetry assumptions. We prove that
$$\Gamma_n = \sup\{\Gamma(\mu): \ \mu \ \mbox{is an isotropic log-concave measure on } \R^n \} \approx n.$$
A key ingredient in both our proofs is a bound due to Eldan and Klartag (2008), which states that
$$\int_{\mathbb{R}^n} |\nabla\psi|\,d\mu \ls Cn$$ 
for every isotropic log-concave probability measure $\mu$ on $\R^n$ with density $e^{-\psi}$. 
We also present further applications of this estimate to projections of log-concave functions projections, moment and surface area measures of isotropic log-concave functions, highlighting the central role of the gradient of the logarithmic potential in high-dimensional convexity.
\end{abstract}

\section{Introduction}\label{section:1}

The classical Brunn--Minkowski inequality asserts that for every pair of nonempty compact sets $K,L$ in $\mathbb{R}^n$ and every $\lambda\in(0,1)$, we have
\begin{equation} \label{eq:brunn-minkowski}
\vol_n(\lambda K+(1-\lambda) L)^{\frac{1}{n}} \gr \lambda\vol_n(K)^{\frac{1}{n}} + (1-\lambda)\vol_n(L)^{\frac{1}{n}},
\end{equation}
where $\vol_n$ denotes the $n$-dimensional Lebesgue measure. A central problem in modern convex geometry is to understand to what extent this dimensional concavity property extends beyond Lebesgue measure.

A natural framework for such extensions is provided by log-concave measures on $\mathbb{R}^n$. Recall that a Borel measure $\mu$ with density $f=e^{-\psi}$ is log-concave if $\psi$ is convex. This class includes, on the one hand, uniform measures on convex bodies and, on the other hand, Gaussian-type distributions, and plays a fundamental role in convex geometry, analysis, and probability. The dimensional Brunn--Minkowski conjecture asks whether the concavity in \eqref{eq:brunn-minkowski} persists for all log-concave measures under symmetry assumptions. More precisely, if $\mu$ is an even log-concave measure on $\mathbb{R}^n$, does one have
\begin{equation}\label{eq:dimBM}
\mu(\lambda K+(1-\lambda) L)^{\frac{1}{n}} \gr \lambda\mu(K)^{\frac{1}{n}}+(1-\lambda) \mu(L)^{\frac{1}{n}}
\end{equation}
for every pair of symmetric convex sets $K,L$ in $\R^n$ and every $\lambda\in(0,1)$?

The conjecture was first formulated by Gardner and Zvavitch \cite{Gardner-Zvavitch-2010} in the Gaussian setting, where it was later confirmed by Eskenazis and Moschidis \cite{Eskenazis-Moschidis-2021}. The general formulation mentioned here is due to Colesanti, Livshyts and Marsiglietti \cite{CLM17}. Further important progress was obtained by Cordero-Erausquin and Rotem \cite{Cordero-Rotem-2023}, who verified \eqref{eq:dimBM} for all rotationally invariant log-concave measures. It is also known that the conjecture follows from the logarithmic Brunn--Minkowski conjecture of B\"or\"oczky, Lutwak, Yang and Zhang \cite{BLYZ12} in every fixed dimension (see \cite{Livshyts-Marsiglietti-Nayar-Zvavitch-2017}), and is therefore valid in dimension $n=2$. More recently, a series of works \cite{AL25,Cordero-Eskenazis-2025,MMRR25} have also investigated functional forms of the conjecture.

\medskip

In Section~\ref{section:3} we revisit an analytic mechanism due to Kolesnikov and E.~Milman \cite{Kolesnikov-Milman-2018, Kolesnikov-Milman-2022}, which has been the primary method towards weighted Brunn--Minkowski-type inequalities for measures on $\R^n$ under central symmetry assumptions since the later work of Kolesnikov and Livshyts \cite{Kolesnikov-Livshyts-2021}. This method reduces \eqref{eq:dimBM} to coercivity estimates for the diffusion operator
$$Lu=\Delta u-\langle \nabla u,\nabla\psi\rangle$$
associated with $\mu$. The key point is that Brunn--Minkowski inequalities can be derived from their infinitesimal versions corresponding to perturbations of convex sets which are expressed via support functions and linearization along Minkowski combinations. This leads to differential inequalities involving curvature and weighted boundary integrals, which are ultimately controlled by spectral $\Gamma_2$-type quantities of the form
$$\int_K \big(\|\nabla^2 u\|^2+\langle \nabla^2\psi\,\nabla u,\nabla u\rangle\big) \,d\mu$$
for solutions of appropriately chosen second-order elliptic equations.

Using this framework, Livshyts \cite{Livshyts-2023} proved a version of \eqref{eq:dimBM} for all even log-concave measures, with exponent $c_n = n^{-4-o_n(1)}$. Our first main result improves this bound.

\begin{theorem}\label{thm:dimBM-main}
Let $\mu$ be an even log-concave probability measure on $\mathbb{R}^n$. Then, for every pair of symmetric convex sets $K,L$ in $\mathbb{R}^n$ and every $\lambda\in (0,1)$,
\begin{equation} \label{equ:dim1}\mu(\lambda K+(1-\lambda) L)^{c_n} \gr \lambda \mu(K)^{c_n}+(1-\lambda)\mu(L)^{c_n},
\end{equation}
where
\begin{equation} c_n \gr \frac{c}{n^3 \ln n}.
\end{equation}
\end{theorem}

Our approach is based on a refined analysis of the mechanism underlying the method of \cite{Kolesnikov-Milman-2018,Kolesnikov-Milman-2022,Livshyts-2023}. A central role is played by the control of the gradient of the potential $\psi$. More precisely, one seeks large subsets on which $|\nabla\psi|$ is bounded, since such bounds translate into coercivity estimates for the operator $L$, as was already observed in \cite{Livshyts-2023}. This leads naturally to the study of the quantity
\begin{equation}\label{eq:main-quantity}
\int_{\mathbb{R}^n} |\nabla \psi(x)|\, d\mu(x),
\end{equation}
for isotropic log-concave probability measures, which is interpreted as the \emph{functional perimeter} of the log-concave function $e^{-\psi}$ (see \cite{Colesanti-Fragala-2013, Cordero-Klartag-2015, Rotem-2022, Rotem-2023}). At the same time, we revisit the argument of Livshyts, clarifying some subtle points and providing the necessary justifications. This yields a complete and self-contained proof of Theorem~\ref{thm:dimBM-main}, together with an improved exponent.

To describe our second main result, for any convex body $A$ in $\mathbb{R}^n$, consider the $\mu$-perimeter of $A$ defined by
$$\mu^+(\partial A)=\liminf_{\epsilon \to 0^+}\frac{\mu(A+\epsilon B_2^n)-\mu(A)}{\epsilon}.$$
If $\mu$ admits a density $f$ with respect to Lebesgue measure, then
$$\mu^+(\partial A)=\int_{\partial A}f(x)\,d\mathcal{H}^{n-1}(x).$$
The \emph{maximal perimeter} of the measure $\mu$ is defined as
$$\Gamma(\mu) = \sup \{ \mu^+(\partial A): \ A \mbox{ is a convex body in } \R^n\}.$$
This quantity has been studied for the Gaussian measure by Ball and Nazarov \cite{Bal93,Naz03} and for rotationally invariant log-concave measures by Livshyts \cite{Liv13,Liv14,Liv15}. In a more recent work \cite{Liv21}, Livshyts studied the maximal perimeter $\Gamma(\mu)$ of isotropic measures $\mu$ establishing a lower bound which, together with the resolution of the thin shell conjecture by Klartag and Lehec \cite{KL25} and the result of Nazarov \cite{Naz03}, yields
$$\gamma_n = \inf\{\Gamma(\mu): \ \mu \ \mbox{is an isotropic log-concave probability measure on } \R^n \} \approx n^{1/4}.$$
For the corresponding supremal quantity
$$\Gamma_n = \sup\{\Gamma(\mu): \ \mu \ \mbox{is an isotropic log-concave probability measure on } \R^n \},$$
an upper bound $\Gamma_n \ls Cn^2$ was established in \cite{Liv21} which was recently improved to $\Gamma_n\ls Cn^{3/2}$ in \cite{BGHT}. We provide 
an asymptotically sharp bound for $\Gamma_n$.

\begin{theorem} \label{thm:weight-per}
For every $n\in\mathbb{N}$, we have $\Gamma_n \approx n$.
\end{theorem}

Both our main results rely upon the following asymptotically sharp upper bound for the $L_1$ norm of the gradient of the logarithmic potential \eqref{eq:main-quantity}, which was first proven in \cite[Lemma~11]{EK08}.

\begin{theorem} [Eldan and Klartag] \label{thm:gradient-upper}
Let $\mu$ be an isotropic log-concave probability measure on $\mathbb{R}^n$ with density $f=e^{-\psi}$. Then,
\begin{equation} \label{eq:ek}
\int_{\mathbb{R}^n} |\nabla \psi(x)|\, d\mu(x) \ls Cn,
\end{equation}
where $C>0$ is an absolute constant.
\end{theorem}

The fact that Theorem \ref{thm:gradient-upper} was proved in \cite{EK08} was brought to our attention by B.~Klartag after the first online posting of the present paper. In Section \ref{section:5} below, we present a new proof of this result which was obtained independently of the work of Eldan and Klartag. Our proof follows a functional and variational approach. The key idea is to study infinitesimal perturbations of $f$ via the Asplund product and dilation, and to identify the first variation both at the pointwise level and after integration. This leads to inequalities connecting $\int |\nabla \psi|\,d\mu$ with entropy-type expressions involving $f$. Combining these estimates with isotropic normalization and convexity arguments yields the desired linear bound for \eqref{eq:main-quantity}. We refer to Section \ref{section:5} for further remarks related to Theorem \ref{thm:gradient-upper}, including lower bounds and improved bounds under additional symmetry (e.g., in the radial case, see Proposition \ref{prop:radial}).

\smallskip

It follows readily from Theorem \ref{thm:gradient-upper} and Markov's inequality that if $\mu$ is an isotropic log-concave probability measure on $\mathbb{R}^n$, then there exists a Borel set $A$ in $\mathbb{R}^n$ with $\mu(A)\gr c>0$ such that
$$|\nabla\psi(x)| \ls Cn \quad \text{for all } x\in A,$$
where $c,C>0$ are absolute constants. A crucial difficulty, which is pertinent to Theorem \ref{thm:dimBM-main}, is that such sets need not be convex, which prevents a direct application of the analytic machinery underlying the approach of \cite{Livshyts-2023}. This reveals a fundamental obstruction: further progress on the dimensional Brunn--Minkowski conjecture requires the construction of large symmetric \emph{convex} subsets on which $|\nabla\psi|$ is well controlled.

\medskip

Our final contribution consists of a series of applications and reformulations of Theorem~\ref{thm:gradient-upper}, highlighting its geometric and functional consequences. A unifying principle is that the quantity $\int |\nabla\psi|\,d\mu$ governs a range of geometric features associated with log-concave functions.

In \S 6.1, Theorem~\ref{thm:gradient-upper} is interpreted in the framework of moment measures. While it is known that the moment measure $\mu_f$ 
of a log-concave function $f$ with finite positive integral has finite first moment under mild assumptions, a fact that plays 
a central role in the characterization of moment measures established by Cordero-Erausquin and Klartag \cite{Cordero-Klartag-2015}, the estimate \eqref{eq:ek} yields 
the sharp quantitative bound
$$\int_{\mathbb{R}^n} |y|\,d\mu_f(y) \ls Cn,$$
for isotropic log-concave functions $f$. 

In \S 6.2, the estimate of Theorem~\ref{thm:gradient-upper} is applied to the study of the pair of surface area measures $(\mu_f,\nu_f)$ associated to a log-concave function, 
in the sense of functional convexity (see \cite{Rotem-2022, Rotem-2023}).  Here $\mu_f$ coincides with the moment measure if $f$ is essentially continuous, while $\nu_f$ captures boundary contributions through the Gauss map of the support of $f$ and appears naturally in first variation formulas of Minkowski type. Building on the previous gradient estimates, we derive  the following general theorem.

\begin{theorem}\label{th:nu-f}Let $\mu_{f}$ and $\nu_{f}$ be the surface area measures of an isotropic log-concave function $f$ on $\mathbb{R}^n$.
Then,
$$\int_{\mathbb{R}^n} |y| \, d\mu_f(y)+\nu_f(S^{n-1})\ls Cn$$
for some absolute constant $C>0$.
\end{theorem}

Theorem \ref{th:nu-f} contains and complements the bound for the moment measure $\mu_f$, providing uniform control on both components of the functional surface area.
This shows that, in the isotropic setting, the total mass of these measures is again of order $n$, reinforcing the idea that 
Theorem~\ref{thm:gradient-upper} captures a fundamental dimensional constraint in the functional extension of classical convex geometric notions. Due to a generalized version of the co-area formula (see \eqref{eq:gen-coarea}), Theorem \ref{th:nu-f} equivalently asserts that 
\begin{equation*}
\int_0^\infty \mathcal{H}^{n-1}(\partial \{x: \ f(x)\gr t\}) \, dt \ls Cn, 
\end{equation*}
for every isotropic log-concave function $f$ (see Corollary \ref{cor:H}). This is crucial for the proof of Theorem \ref{thm:weight-per}.

Finally, in \S 6.3, the focus shifts to projections of log-concave functions. For an isotropic log-concave density $f$ on $\mathbb{R}^n$, we consider its projections onto 
hyperplanes and prove that the average $L^1$-norm of these projections is of order $\sqrt{n}$, which is optimal with respect to the dimension. 
This result can be viewed as a functional counterpart of classical estimates for volumes of projections of convex bodies, and is obtained by combining gradient bounds with Cauchy’s surface area formula and a geometric \mbox{interpretation of level sets.}

\medskip

The paper is organized as follows. In Section~\ref{section:2} we collect background material on convex bodies and isotropic log-concave measures. In Section~\ref{section:3} we develop the analytic approach to the dimensional Brunn--Minkowski problem and prove Theorem~\ref{thm:dimBM-main}. In Section \ref{section:4} we prove Theorem \ref{thm:weight-per}. In Section~\ref{section:5} we present a new proof of the gradient estimate of Theorem~\ref{thm:gradient-upper} and related remarks. Finally, in Section~\ref{section:6} we present applications and further geometric consequences of these main bounds.

\medskip

\noindent {\bf Acknowledgements.} We would like to thank Bo'az Klartag, Emanuel Milman, Eli Putterman and Michael Roysdon for helpful pointers to the literature and constructive feedback.

\section{Background information and auxiliary results}\label{section:2}

We work in $\mathbb{R}^n$, equipped with the standard inner product $\langle \cdot, \cdot \rangle$.  
The associated Euclidean norm is denoted by $|\cdot|$, the Euclidean unit ball by $B_2^n$, and the Euclidean unit sphere by $S^{n-1}$.  
Lebesgue measure in $\mathbb{R}^n$ is denoted by $\vol_n$, and we write $\omega_n = \vol_n(B_2^n)$ for the volume of the Euclidean unit ball.
We denote by $\sigma$ the rotationally invariant probability measure on $S^{n-1}$. The Grassmann manifold $G_{n,k}$ of $k$-dimensional subspaces 
of ${\mathbb R}^n$ is equipped with the Haar probability measure $\nu_{n,k}$. 

Throughout the paper, the symbols $C, c, c', c_1, c_2, \ldots$ denote absolute positive constants whose values may change from line to line.  
Whenever we write $a \approx b$, we mean that there exist absolute constants $c_1, c_2 > 0$ such that $c_1 a \ls b \ls c_2 a$. 

\bigskip 

\noindent \textbf{\S 2.1. Convex bodies.} A convex body in $\mathbb{R}^n$ is a compact convex set $K$ with nonempty interior.  
It is called symmetric if $K = -K$, and centered if its barycenter $\operatorname{bar}(K) = \frac{1}{\vol_n(K)} \int_K x\,dx$ is at the origin. 
For every convex body $K$ in $\mathbb R^n$ we denote by $\overline{K}$ the homothetic copy of $K$ scaled to have unit volume,
namely $\overline{K}:=\vol_n(K)^{-1/n}K$.

Let $K$ be a convex body in $\mathbb{R}^n$ with $0\in {\rm int}(K)$. 
The radial function of $K$ is defined by $\varrho_K(x) = \max\{ t > 0 : t x \in K \}$ for all nonzero $x$, and the support function of $K$ is given by 
$h_K(x)=\max\{\langle x,y\rangle :y\in K \}$ for all $x\in\mathbb{R}^n$. The Minkowski functional
of $K$ is defined by $\|x\|_K = \inf\{t>0:\ x\in tK\}$ for all $x\in\mathbb{R}^n$. If $K$ is symmetric then $\|\cdot\|_K$ is a norm on $\mathbb{R}^n$.

The polar body $K^{\circ }$ of a convex body $K$ in ${\mathbb R}^n$ with $0\in {\rm int}(K)$ is the convex body
\begin{equation*}
K^{\circ}:=\bigl\{y\in {\mathbb R}^n: \langle x,y\rangle \ls 1\;\hbox{for all}\; x\in K\bigr\}.
\end{equation*}
The surface area $S(K)$ of $K$ is defined by
$$S(K):=\mathcal{H}^{n-1}(\partial K),$$
where $\mathcal{H}^{n-1}$ denotes the $(n-1)$-dimensional Hausdorff measure on the boundary $\partial K$ of $K$. 

A convex body $K$ in ${\mathbb R}^n$ is called isotropic if it has volume $1$, is centered, and its covariance matrix is a multiple of the identity. 
Equivalently, there exists a constant $L_K>0$, called the isotropic constant of $K$, such that
\begin{equation*}
\|\langle \cdot ,\xi\rangle\|_{L_2(K)}^2:=\int_K\langle x,\xi\rangle^2dx =L_K^2\quad\text{for all}\;\xi\in S^{n-1}.
\end{equation*}
We shall use a number of geometric properties of isotropic convex bodies. For instance, it is known that $L_K\gr L_{B_2^n}\gr c$ for an absolute constant $c>0$. A proof of this assertion may be found e.g. in \cite[Chapter~3]{BGVV-book}.

Bourgain's slicing problem \cite{Bourgain-1986} asks whether there exists an absolute constant $C>0$ such that
\begin{equation}\label{eq:conjecture}
L_n:=\max\{ L_K:K\ \hbox{is an isotropic convex body in}\ \mathbb{R}^n\}\ls C.
\end{equation}
An affirmative solution was recently obtained by Klartag and Lehec~\cite{Klartag-Lehec-2025}, following an important
contribution by Guan~\cite{Guan-preprint} (see also~\cite{Bizeul-2025} for an alternative proof). 
Consequently, $L_K\approx 1$, uniformly in $n$, for every isotropic convex body $K$ in $\mathbb{R}^n$.
For further background, we refer to the survey~\cite{Giannopoulos-Pafis-Tziotziou-2025}.

\bigskip 

\noindent \textbf{\S 2.2. Log-concave functions.} A function $f:\mathbb R^n \rightarrow [0,\infty)$ is called log-concave if its support $K_f=\overline{\{f>0\}}$ 
is a convex set in ${\mathbb R}^n$ and the restriction of $\ln{f}$ to it is concave. 
We denote by ${\rm LC}_{n}$ the class of all upper semi-continuous log-concave functions. Note
that the class of convex bodies in $\mathbb{R}^{n}$ embeds naturally into ${\rm LC}_{n}$ using the map 
$K\mapsto \mathds{1}_{K}$. 

A log-concave function $f\in {\rm LC}_n$ is essentially continuous if the set of its discontinuity points has zero measure 
for the $(n-1)$-dimensional Hausdorff measure. This is equivalent to the assumption that $f\in W^{1,1}(\mathbb{R}^n)$ (see \cite[Remark~3.3]{Rotem-2022}). 
For every $f\in {\rm LC}_n$ with $0<\int f<\infty$ we have
\begin{equation} \label{eq:gen-coarea}
\int_{0}^{\infty} \mathcal H^{n-1}(\partial\{x:f(x)\gr t\})\,dt=\int_{\mathbb{R}^n}|\nabla f(x)|\,dx+\int_{\partial K_f}f(x)\,d\mathcal{H}^{n-1}(x)
\end{equation}
(see \cite[Theorem~3.2]{Rotem-2022}). In particular, $f$ is essentially continuous if and only if we have the co-area formula
\begin{equation}\label{eq:coarea}\int_{\mathbb{R}^n}|\nabla f(x)|\,dx= \int_{0}^{\infty} \mathcal H^{n-1}(\partial\{x:f(x)\gr t\})\,dt.\end{equation}
We emphasize that \cite[Theorem~3.2]{Rotem-2022} in fact contains a misprint, as the term $\mathcal H^{n-1}(\partial\{x:f(x)\gr t\})$ in the co-area formula is replaced by $\mathcal H^{n-1}(\{x:f(x) = t\})$. We confirmed with the author of \cite{Rotem-2022} that this was due to an erroneous reproduction of a formula from \cite[Section~5.5]{Evans}. It does not affect any further results from that paper and shall be used in this form in the sequel.

Let $f=e^{-\psi}\in {\rm LC}_n$. The Legendre transform of $\psi$ is
$$\mathcal{L}\psi(x)= \sup_{y\in\mathbb{R}^n}\{ \langle x,y\rangle - \psi(y)\}.$$
It is always a convex, lower semi-continuous function, and satisfies the involution property 
$\mathcal{L}(\mathcal{L}\psi)=\psi$ if $\psi$ is lower semi-continuous, convex and proper (meaning
that $\mathrm{dom}(\psi) := \{x \in \mathbb{R}^n : \psi(x) < \infty\} \neq \varnothing$). 

Given two log-concave functions $f=e^{-\psi}$ and $g=e^{-\varphi}$ in  ${\rm LC}_{n}$, we define the sup-convolution or 
Asplund product of $f$ and $g$ by
$$(f\star g)(x)= \sup_{y\in\mathbb{R}^n} f(y)\, g(x-y)= \exp\!\big(-(\psi \square \varphi)(x)\big),$$
where the inf-convolution of two convex functions $\psi$ and $\varphi$ is
$$(\psi \square \varphi)(x)=\inf_{y\in\mathbb{R}^n}\big\{\psi(y)+\varphi(x-y)\big\}.$$
For $t>0$ and a log-concave function $f \in {\rm LC}_{n}$, the functional dilation is defined by
$$(t \cdot f)(x) = f\big(x/t\big)^t.$$
This transformation respects log-concavity and is a natural functional counterpart of geometric dilation.

\bigskip 

\noindent \textbf{\S 2.3. Isotropic log-concave probability measures.} We say that a Borel probability measure $\mu$ on $\mathbb R^n$ is log-concave 
if $\mu(H)<1$ for every hyperplane $H$ in ${\mathbb R}^n$ (we then say that $\mu$ is full-dimensional) and $\mu(\lambda A+(1-\lambda)B) \gr \mu(A)^{\lambda}\mu(B)^{1-\lambda}$ for any pair of compact sets $A,B$ in ${\mathbb R}^n$ and any $\lambda \in (0,1)$. Borell \cite{Borell-1974} has proved that, under these assumptions, $\mu $ has a log-concave density $f$. The Brunn-Minkowski inequality implies that if $K$ is a convex body in $\mathbb R^n$ then the indicator function $\mathds{1}_{K} $ of $K$ is the density of a log-concave measure, the Lebesgue measure on $K$.

Let $f:\mathbb{R}^n\to [0,\infty)$ be a log-concave function with finite, positive integral. Its barycenter is defined by
$$\operatorname{bar}(f)= \frac{\int_{\mathbb{R}^n} x\, f(x)\, dx}{\int_{\mathbb{R}^n} f(x)\, dx}.$$
We say that $f$ is centered if $\operatorname{bar}(f)=0$.
We shall use the following result of Fradelizi \cite{Fradelizi-1997}: if $f$ is a centered log-concave density on ${\mathbb R}^n$, then
\begin{equation}\label{eq:frad-2}\|f\|_{\infty }\ls e^nf(0).\end{equation}
The isotropic constant of a log-concave function $f$ with finite positive integral is the affine-invariant quantity
\begin{equation}\label{eq:definition-isotropic}
L_f:= \left( \frac{\|f\|_{\infty}}{\int_{\mathbb{R}^n} f(x)\, dx} \right)^{1/n} \det(\operatorname{Cov}(f))^{1/(2n)},
\end{equation}
where $\operatorname{Cov}(f)$ denotes the covariance matrix of $f$. A log-concave function $f$ is called isotropic if
$$\operatorname{bar}(f)=0,\quad \int_{\mathbb{R}^n}f(x)\,dx=1,\quad \text{and} \quad \operatorname{Cov}(f)=I_n.$$
In this case, $L_f=\|f\|_{\infty}^{1/n}$. A full-dimensional log-concave probability measure $\mu$ on $\mathbb{R}^n$ is called isotropic if its
density $f$ is isotropic. Then, we set $L_{\mu}:=L_f$.

Note that a centered convex body $K$ in $\mathbb{R}^n$ with $\vol_n(K)=1$ is isotropic if and only if the log-concave function $L_K^n\mathds{1}_{K/L_K}$ is isotropic.

Let $\mu$ be a full-dimensional log-concave probability measure on $\mathbb{R}^n$.
For any $1\ls k \ls n-1$ and any $k$-dimensional subspace $F$ of $\mathbb{R}^n$, the marginal of $\mu$ onto $F$ is defined by
$$\pi_F(\mu)(B):=\mu(P_F^{-1}(B)),$$
for every Borel set $B\subseteq F$. The measure $\pi_F(\mu)$ is log-concave and admits a density
$$(\pi_F f)(x)=\int_{F^\perp} f(y+x)\,dy.$$
If $f$ is centered (respectively isotropic), then so is $\pi_F f$
(see~\cite[Proposition~5.1.11]{BGVV-book}). In particular, if $\mu$ is isotropic and $F_{\xi}=\{t\xi:t\in\mathbb{R}\}$ for $\xi\in S^{n-1}$, 
then the one-dimensional marginal $$g_{\xi}(t)=(\pi_{F_{\xi}}f)(t)=\int_{F_{\xi}^{\perp}}f(y+t\xi)\,dy$$
is an isotropic log-concave density on $\mathbb{R}$. Consequently, as shown e.g. in \cite{Fradelizi-1997},
\begin{equation}\label{eq:dim-1}\|g_{\xi}\|_{\infty}=L_{g_{\xi}}\ls 1.\end{equation}
It is known that every centered log-concave density $f$ admits an isotropic position: there exists $T\in GL_n$ such that the push-forward density
$$f_T(x)=\frac{1}{|\det T|}\,f(T^{-1}x)$$
is isotropic (see \cite[Section~2.3]{BGVV-book}).
Moreover, $f_T$ is also log-concave, and $L_{f_T}=L_f$. It is also known (see \cite[Proposition~2.3.12]{BGVV-book}) that $L_f\gr c$ for every
isotropic log-concave function $f$ on $\mathbb{R}^n$, where $c>0$ is an absolute constant. On the other hand, Ball~\cite{Ball-1988} proved that for every $n$,
$$\tilde{L}_n:=\sup\big\{L_f: f\;\text{is a log-concave density on $\mathbb{R}^n$}\big\}\ls C_1L_n,$$
and hence $\tilde{L}_n\ls C_2$ by the affirmative solution of Bourgain's slicing problem. 

We refer to~\cite{AGA-book,AGA-book-2} for asymptotic convex geometry, and to~\cite{BGVV-book} for background 
on isotropic convex bodies and log-concave measures.

\bigskip 

\noindent \textbf{\S 2.4. Level sets of the density.} Let $\mu$ be a centered log-concave probability measure on $\mathbb{R}^n$ with density 
$f=\exp(-\psi)$, where $\psi$ is a convex function. For every $t\gr 0$ we consider the convex set
$$R_t(\mu )=\{x\in {\mathbb R}^n:f(x)\gr e^{-t}f(0)\}.$$
Using the log-concavity of $f$ we easily check that $R_t(\mu )$ is convex. Note also that $0\in {\rm int}(R_t(\mu ))$ for every $t>0$.
To show that $R_t(\mu)$ is bounded, we recall that since $f$ is log-concave and has finite positive integral we
have that there exist constants $A,B>0$ such that
\begin{equation}\label{eq:A-B}f(x)\ls Ae^{-B|x|}\end{equation}
for all $x\in {\mathbb R}^n$ (see \cite[Lemma~2.2.1]{BGVV-book}). Therefore, if $x\in R_t(\mu)$ we get
that $|x|\ls\frac{1}{B}\big(\ln (A/f(0))+t\big)$. Another consequence of \eqref{eq:A-B} is that $f$ has
finite moments of all orders.


The next proposition, which is essentially due to Klartag (see \cite[Lemma~5.2]{Klartag-2007} and \cite{Giannopoulos-Tziotziou-2025}
for the precise form below) shows that the measure of $R_t(\mu)$ increases to $1$
exponentially fast as $t\to\infty $.

\begin{proposition}\label{prop:r-2}For every $t\gr 3n$ we have that $\mu(R_t(\mu ))\gr 1-e^{-t/4}$.
\end{proposition}

\begin{proof}Consider the convex function $\varphi$ defined by $e^{-\varphi(x)}=f(x)/f(0)$. Note that $\varphi(0)=0$.
Then
\begin{align*}
\int_{\mathbb{R}^n}e^{\varphi(x)/2}d\mu(x) &=f(0)\int_{\mathbb{R}^n}e^{-\varphi(x)/2}dx=f(0)\int_{\mathbb{R}^n}e^{-(\varphi(x)+\varphi(0))/2}dx\\
&\ls f(0)\int_{\mathbb{R}^n}e^{-\varphi(x/2)}dx=2^nf(0)\int_{\mathbb{R}^n}e^{-\varphi(x)}dx=2^n.
\end{align*}
For any $t>0$ we have $R_t(\mu)=\{x:\varphi(x)\ls t\}$. From Markov's inequality we get
$$1-\mu(R_t(\mu))=\mu(\{x:\varphi(x)>t\})\ls e^{-t/2}\int_{\mathbb{R}^n}e^{\varphi(x)/2}d\mu(x)\ls 2^ne^{-t/2}.$$
If $t\gr (4\ln 2)n$ then $2^ne^{-t/2}\ls e^{-t/4}$, and this implies that
$$\mu(R_t(\mu))\gr 1-e^{-t/4},\qquad\text{for all}\;t\gr (4\ln 2)n.$$
The result follows. \end{proof}

We shall also use the fact that if $\mu$ is isotropic and $t$ is large enough then $R_t(\mu)$ contains a constant multiple of the Euclidean unit ball.
The proof of the next lemma is essentially contained in \cite[Lemma~5.4]{Klartag-2007}. 

\begin{lemma}\label{lem:gamma-ball}Let $\mu$ be an isotropic log-concave probability measure on ${\mathbb R}^n$.
For any $n\gr 10$ and any $t\gr 3n$ we have that
$$R_t(\mu)\supseteq  \tfrac{1}{3}B_2^n.$$
\end{lemma}

A proof of Lemma~\ref{lem:gamma-ball} in the exact form stated above can be found in \cite[Lemma~3.4]{Giannopoulos-Tziotziou-2025} and \cite[Lemma~3.2]{BGHT}.

\bigskip 

\noindent \textbf{\S 2.5. Stability of the Poincar\'{e} constant.} Let $\mu $ be a Borel probability measure on ${\mathbb R}^n$. 
Recall that the Cheeger constant $\chi_{\mu }$ of $\mu $ is the largest constant $c\gr 0$ for which we have
\begin{equation}\label{eq:cheeger-1}\mu^+(\partial E)\gr c\,\min\{\mu (E),1-\mu (E)\}\end{equation}
for every Borel subset $E$ of ${\mathbb R}^n$. The reciprocal Cheeger constant of $\mu$ is $\psi_{\mu}:=1/\chi_{\mu}$. 

We also say that $\mu $ satisfies the Poincar\'{e} inequality with constant $\vartheta >0$ if
\begin{equation}\label{eq:poincare-1}{\rm Var}_\mu(f)\ls \vartheta^2\int |\nabla f|^2\, d\mu,
\end{equation}for all smooth functions $f$ on ${\mathbb R}^n$, where ${\rm Var}_\mu(g)={\mathbb E}_\mu(g^2)-({\mathbb E}_\mu(g))^2$ is the
variance of $g$ with respect to $\mu$. The Poincar\'{e} constant $\vartheta_{\mu }$ of $\mu $ is the smallest constant 
$\vartheta>0$ for which \eqref{eq:poincare-1} is satisfied.

It is known (see~\cite[Theorem~2.3.1]{AGA-book-2} that if $\alpha(\mu)$ is the smallest constant $\alpha>0$ with the property
that every integrable, locally Lipschitz function $f:{\mathbb R}^n\to {\mathbb R}$ satisfies
\begin{equation}\label{eq:cheeger-poincare-1-1}\int_{{\mathbb R}^n} |f(x)-{\mathbb E}_{\mu }(f)|\,d\mu (x)\ls
\alpha\int_{{\mathbb R}^n} |\nabla f (x)|\,d\mu (x),\end{equation}
then, $\psi_{\mu}/2\ls \alpha(\mu)\ls 2\psi_{\mu}$. It is also known (see~\cite{Cattiaux-Guillin-2020}) that
$$\frac{1}{2}\int_{\mathbb{R}^n}|f(x)-\mathbb{E}_{\mu}(f)|\,d\mu(x)\ls \int_{\mathbb{R}^n}|f(x)-m_{\mu}(f)|\,d\mu(x)\ls 
\int_{\mathbb{R}^n}|f(x)-\mathbb{E}_{\mu}(f)|\,d\mu(x)$$
for all $f$, where $m_{\mu}(f)$ is the median of $f$ with respect to $\mu$. It follows that if
$\beta(\mu)$ is the smallest constant $\beta>0$ with the property
that for every integrable, locally Lipschitz function $f:{\mathbb R}^n\to {\mathbb R}$,
\begin{equation}\label{eq:cheeger-poincare-1-2}\int_{{\mathbb R}^n} |f(x)-m_{\mu }(f)|\,d\mu (x)\ls
\beta\int_{{\mathbb R}^n} |\nabla f (x)|\,d\mu (x),\end{equation}
then, $\alpha(\mu)/2\ls \beta(\mu)\ls \alpha(\mu)$. Combining the above, we see that
\begin{equation}\label{eq:cheeger-poincare-1-3}\psi_{\mu}/4\ls \beta(\mu)\ls 2\psi_{\mu}.\end{equation}
A theorem of Maz'ya \cite{Mazya-1960}, \cite{Mazya-1962} and Cheeger \cite{Cheeger-1970}
shows that the Poincar\'{e} constant is bounded by the reciprocal Cheeger constant: If $\mu$ is a Borel probability measure with 
reciprocal Cheeger constant $\psi_\mu$ then its Poincar\'e constant $\vartheta_{\mu}$ satisfies
\begin{equation}\label{eq:cheeger-poincare-1}\vartheta_{\mu}\ls 2\psi_{\mu }.\end{equation}
On the other hand, the assumption that $\mu$ is log-concave implies a reverse inequality with a constant that
does not depend on the dimension. Buser \cite{Buser-1982} (see also Ledoux \cite{Ledoux-1994})
has shown that if $\mu$ is a log-concave probability measure on $\mathbb{R}^n$, then 
\begin{equation}\label{eq:cheeger-poincare-2}\psi_{\mu }\ls c\,\vartheta_{\mu},\end{equation}
where $c>0$ is an absolute constant.

E.~Milman \cite[Theorem~5.5]{EMilman-2009} has shown that the ratio of the Cheeger constants of two log-concave probability measures $\mu$ and $\nu$
on $\mathbb{R}^n$ with densities $f$ and $g$ is controlled by their total variation distance
$$d_{{\rm TV}}(\mu,\nu)=\frac{1}{2}\int_{\mathbb{R}^n}|f(x)-g(x)|\,dx.$$
More precisely, if $d_{{\rm TV}}(\mu,\nu)= 1-\varepsilon$ for some $\varepsilon\in (0,1)$ then
\begin{equation}\label{eq:poincare-stability}\chi_{\mu}\ls \frac{c_1}{\varepsilon^2}\max\{1,\ln(1/\varepsilon)\}\,\chi_{\nu},\end{equation}
where $c_1>0$ is an absolute constant. An alternative proof of \eqref{eq:poincare-stability} is given by Cattiaux and Guillin in
\cite[Theorem~9.3.10]{Cattiaux-Guillin-2020}. When $\nu$ is itself a restriction $\mu_A$ of $\mu$ on a Borel set $A$, then we have
\begin{align*}
d_{{\rm TV}}(\mu,\mu_A) &=\frac{1}{2}\left(\int_A\left|f(x)-\frac{1}{\mu(A)}f(x)\right|\,dx+\int_{A^c}f(x)\,dx\right)\\
&=\frac{1}{2}\left(\left(\frac{1}{\mu(A)}-1\right)\mu(A)+\mu(A^c)\right)=1-\mu(A).
\end{align*}
In this case, the dependence on $\e=\mu(A)$ in \eqref{eq:poincare-stability} can be improved to logarithmic as shown in \cite[Lemmas~5.1 and 5.4]{EMilman-2009}, though this improved dependence will be immaterial for us.

\begin{proposition} [E.~Milman]\label{prop:theta-stability}Let $\mu $ be a log-concave probability measure on $\mathbb{R}^n$. For every convex
set $A\subseteq\mathbb{R}^n$ with $\mu(A)>0$ we have that
\begin{equation}\label{eq:theta-stability-1}\vartheta_{\mu_A}\ls C \max\left\{1,\ln\left(\frac{1}{\mu(A)}\right)\right\}\,\vartheta_{\mu},\end{equation}
where $\mu_A$ is the restriction of $\mu$ onto $A$ and $C>0$ is an absolute constant. 
\end{proposition}

The Kannan--Lov\'{a}sz--Simonovits conjecture asks if there exists an absolute constant $C>0$ such that
\begin{equation*}\psi_n:=\sup\{ \psi_{\mu }:\mu\ \hbox{is an isotropic log-concave measure on}\ {\mathbb
R}^n\}\ls C.\end{equation*}
In view of the discussion above, an equivalent way to formulate the KLS conjecture is to ask that the Poincar\'{e} inequality
holds for every isotropic log-concave probability measure $\mu$ on ${\mathbb R}^n$ with a constant
that does not depend on the measure or the dimension $n$. The best known result on this problem is due to Klartag \cite{Klartag-2023} and 
provides an (almost) affirmative answer. For any $n\gr 2$ we have that
$$\psi_n\ls C\sqrt{\ln n},$$
where $C>0$ is an absolute constant.

Combining Klartag's estimate with Proposition~\ref{prop:theta-stability} we immediately obtain the next fact that will be useful in the proof of the dimensional Brunn--Minkowski inequality.

\begin{proposition}\label{prop:theta-stability-isotropic}Let $\mu $ be an isotropic log-concave probability measure on $\mathbb{R}^n$. For every $\alpha\in (0,1)$
and every convex set $A\subseteq\mathbb{R}^n$ with $\mu(A)\gr \alpha$ we have that
$$\vartheta_{\mu_A}\ls C(\alpha)\sqrt{\ln n},$$
where $\mu_A$ is the restriction of $\mu$ onto $A$ and $C(\alpha)>0$ is a constant that depends only on $\alpha$. 
\end{proposition}

It should be noted that a reverse inequality to \eqref{eq:theta-stability-1} holds for Borel sets $A$ that have large measure without further convexity assumptions. It is proved in \cite[Proposition~9.2.5]{Cattiaux-Guillin-2020} that if $\mu$ is a log-concave probability measure such that
$$\int_{\mathbb{R}^n} |f(x)-\mathbb{E}_{\mu}(f)|\,d\mu(x)\ls c(t)\int_{\mathbb{R}^n} |\nabla f(x)|\,d\mu(x)+ t\,\mathrm{osc}(f)$$
for some $0<t<1/2$, some constant $c(t)>0$, and all Lipschitz functions $f$, where $\mathrm{osc}(f)=\sup f-\inf f$
denotes the oscillation of $f$, then
$$\beta(\mu)\ls \frac{c(t)}{1-2t}.$$
Let $A$ be a Borel set in $\mathbb{R}^n$ and let $\mu_A$ denote the restriction of $\mu$ onto $A$ as above. Then it is clear that
$$\int_{\mathbb{R}^n} |f(x)-\mathbb{E}_\mu(f)|\,d\mu(x)\ls\int_A |f(x)-\mathbb{E}_\mu(f)|\,d\mu(x)+(1-\mu(A))\,\mathrm{osc}(f),$$
therefore
$$\int_{\mathbb{R}^n} |f(x)-\E_{\mu}(f)|\,d\mu(x)\ls\mu(A)\beta(\mu_A)\int_{\mathbb{R}^n} |\nabla f(x)|\,d\mu_A(x)+(1-\mu(A))\,\mathrm{osc}(f).$$
This implies the following result.

\begin{proposition}\label{prop:theta-stability-reverse}Let $\mu $ be a log-concave probability measure on $\mathbb{R}^n$. For every Borel 
set $A\subseteq\mathbb{R}^n$ with $\mu(A)>\tfrac{1}{2}$ we have that
$$\beta(\mu)\ls \frac{\mu(A)}{2\mu(A)-1}\beta(\mu_A),$$
and hence
\begin{equation}\label{eq:theta-stability-reverse}\psi_{\mu}\ls  \frac{8\mu(A)}{2\mu(A)-1}\,\psi_{\mu_A},\end{equation}
where $\mu_A$ is the restriction of $\mu$ onto $A$. 
\end{proposition}

\bigskip

\section{The dimensional Brunn--Minkowski inequality}\label{section:3}

In what follows, $\mu$ is a probability measure on $\mathbb{R}^n$ with an even twice continuously differentiable density $f=\exp(-\psi)$.
We also consider the operator
$$Lu=\Delta u-\langle \nabla u,\nabla\psi\rangle $$
for twice continuously differentiable $u:\mathbb{R}^n\to \mathbb{R}$.
We denote by $\mathcal{C}^2$ the class of $C^2$-smooth symmetric convex bodies; these are the convex bodies whose support function is 
twice continuously differentiable on the unit sphere. If $K$ is a $C^2$-smooth symmetric convex body 
in $\mathbb{R}^n$, we write $n_x$ for the normal vector at the point $x\in\partial K$.
We also denote by $\mathcal{C}^2_+$ the class of convex bodies in $\mathcal{C}^2$ with positive Gaussian curvature. We say that a function 
$h:S^{n-1}\rightarrow \mathbb{R}$ is a function in $C^2_+(S^{n-1})$ if it is the support function of a convex body in $\mathcal{C}^2_+$. 

Following Livshyts \cite{Liv24}, we define the concavity power $p(\mu,K)$ of a symmetric convex body $K \in \mathcal{C}_+^2$ with respect to $\mu$ to be the largest $p\gr 0$ such that for every symmetric convex body $L\in\mathcal{C}_+^2$, we have
\begin{equation} \label{eq:conc-power}
\frac{d^2}{d \lambda^2} \Bigg|_{\lambda=1} \mu(\lambda K +(1-\lambda)L)^{p} \ls  0.
\end{equation}
We shall need the following straightforward fact.

\begin{lemma} \label{lem:posi}
For every log-concave probability measure $\mu$ on $\R^n$ with an even twice continuously differentiable density, every symmetric convex body $K\in\mathcal{C}_+^2$ and every invertible linear operator $T:\R^n\to\R^n$,
$$p(\mu,K) = p(T_\ast\mu, TK),$$
where $T_\ast\mu$ is the push-forward of $\mu$ under $T$ given by $(T_\ast\mu)(E)=\mu(T^{-1}E)$, where $E\subseteq\R^n$ is a Borel set.
\end{lemma}

\begin{proof}
Let $p=p(T_\ast\mu, TK)$ and consider a symmetric convex body $L\in\mathcal{C}_+^2$. Then,
$$\frac{d^2}{d\lambda^2}\Bigg|_{\lambda=1} \mu(\lambda K+(1-\lambda)L)^p = \frac{d^2}{d\lambda^2}\Bigg|_{\lambda=1} (T_\ast\mu)\big(\lambda TK + (1-\lambda)TL\big)^p \ls  0$$
by the definition of $p(T_\ast\mu, TK)$, so $p(\mu,K) \gr  p(T_\ast\mu,TK)$. The converse inequality follows similarly.
\end{proof}

A standard local-to-global principle (see \cite[Section 3]{Kolesnikov-Milman-2022} or \cite[Lemma~3.1]{Kolesnikov-Livshyts-2021}) implies that for a given $\mu$ and $p>0$,
\begin{equation} \label{eq:BM-E-1}
\inf_{K\in \mathcal{C}_+^2} \ p(\mu,K) \gr  p
\end{equation}
if and only if for every pair of symmetric convex sets $K,L$ in $\R^n$ and every $\lambda\in[0,1]$, we have
\begin{equation} \label{eq:BM-E-2}
\mu(\lambda K+(1-\lambda)L)^p \gr  \lambda \mu(K)^p + (1-\lambda)\mu(L)^p.
\end{equation}
Indeed, the derivation of \eqref{eq:BM-E-1} from \eqref{eq:BM-E-2} is immediate from the statement of \cite[Lemma~3.1]{Kolesnikov-Livshyts-2021}. Conversely, to show that \eqref{eq:BM-E-1} implies \eqref{eq:BM-E-2}, we apply \cite[Lemma~3.1]{Kolesnikov-Livshyts-2021} to the class $\mathcal{F}$ of all symmetric $\mathcal{C}_+^2$ convex bodies. Note that equation (21) of \cite{Kolesnikov-Livshyts-2021} requires the validity of $\frac{d^2}{d s^2} \mu(K_s)^p \ls 0$ for any one-parameter family $\{K_s\}_{|s|<\e}$ of symmetric convex sets that arise from Wulff shape perturbations of $K$ but an inspection of the proof of this implication reveals that Minkowski convex combinations are sufficient to yield the conclusion (since the perturbation in the last paragraph of \cite[Proof of Lemma~3.1]{Kolesnikov-Livshyts-2021} is chosen to be $\psi=h_L-h_K$). Then, having proven the Brunn--Minkowski inequality \eqref{eq:BM-E-2} for every pair of symmetric convex bodies $K,L \in \mathcal{C}_+^2$, we extend it to all symmetric convex sets by approximation.

In \cite{Kolesnikov-Milman-2018}, Kolesnikov and E.~Milman developed a powerful method for proving lower bounds for concavity powers by appropriate uses of curvature and integration by parts along the lines of H\"ormander's $L_2$ method (see also \cite[Section~2]{Cordero-Eskenazis-2025}). Their main result, as presented in \cite[Proof of Lemma~2.3]{Kolesnikov-Livshyts-2021}, reads as follows.

\begin{theorem}\label{th:KL}
Let $\mu$ be an even measure with a twice continuously differentiable density $f=\exp(-\psi)$ and $K$ be a $\mathcal{C}_+^2$ convex body in $\R^n$. Suppose that every $u\in C^2(K)$ with $Lu=1$ on $K$ satisfies
\begin{equation} \label{eq:BM-E-3}
\frac{1}{\mu(K)}\int_K \big(\|\nabla^2 u\| + \langle \nabla^2\psi \nabla u, \nabla u\rangle\big)\,d\mu \gr  p,
\end{equation}
where $\|\nabla^2u\|$ is the Hilbert--Schmidt norm of the Hessian of $u$. Then, the concavity power of $K$ with respect to $\mu$ satisfies $p(\mu,K)\gr  p$.
\end{theorem}

We emphasize that \cite[Lemma~2.3]{Kolesnikov-Livshyts-2021} presents a similar characterization for the quantity $\inf_{K\in\mathcal{C}_+^2} p(\mu,K)$ via the equivalence of \eqref{eq:BM-E-1} and \eqref{eq:BM-E-2} as the differential inequality \eqref{eq:BM-E-3} is required to hold \emph{for every} symmetric convex set $K$. However, an inspection of the proof readily reveals that the criterion in fact holds for each individual $K$ as well, thus offering a bound for each concavity power $p(\mu, K)$. To see this, observe that one can use \cite[Proposition~3.2]{Kolesnikov-Livshyts-2021} to write the condition \eqref{eq:conc-power} defining the exponent $p(\mu,K)$ as a weighted Poincar\'e-type inequality for functions $f$ defined on the boundary of $K$. Then, one can use Kolesnikov and E.~Milman's \cite{Kolesnikov-Milman-2018} weighted version of the Reilly formula \cite[Proposition~3.4]{Kolesnikov-Livshyts-2021} to rewrite this inequality in terms of the solution of the equation $Lu=1$ with Neumann boundary data $f$. In the case of Gaussian measure, this was also explicitly exploited in \cite[Proof of Theorem A]{Liv24}.

Livshyts \cite{Livshyts-2023} obtains the following estimate, which can be combined with Theorem~\ref{th:KL}. The proof below clarifies a few
subtle points of the original argument.

\begin{theorem}[Livshyts]\label{th:L-reduction}Let $\nu$ be an even log-concave probability measure on $\mathbb{R}^n$. Let $K$ be a
symmetric convex set in $\mathbb{R}^n$ and let $u:K\to\mathbb{R}$ be an even function in $C^2(K)$. Then, for any symmetric
convex set $A\subseteq K$ we have
$$\frac{1}{\nu(K)}\int_K\|\nabla^2u\|^2d\nu \gr \frac{\nu(A)}{\nu(K)}\cdot\frac{\left(\frac{1}{\nu(A)}\int_ALu\,d\nu\right)^2}{n+\frac{1}{\nu(A)}\int_A(\vartheta_{\nu_A}^2|\nabla \psi|^2-2\langle \nabla\psi,x\rangle )\,d\nu}.$$
\end{theorem}

\begin{proof}We start with the inequality
\begin{equation}\label{eq:first}\int_K\|\nabla^2u\|^2d\nu\gr \int_A\|\nabla^2u\|^2d\nu .\end{equation}
For every $t\in\mathbb{R}$ we define $v_t(x)=u(x)-\frac{t}{2}|x|^2$. Direct computation shows that
\begin{equation}\label{eq:L-18}\|\nabla^2u\|^2=\|\nabla^2v_t\|^2+2t\Delta v_t+t^2n\end{equation}
and
\begin{equation}\label{eq:L-19}(Lu)(x)=(Lv_t)(x)+tL(|x|^2/2)=(Lv_t)(x)+tn-t\langle x,\nabla\psi(x)\rangle .\end{equation}
Therefore,
\begin{equation}\label{eq:L-20}\Delta v_t=\langle\nabla\psi ,\nabla v_t\rangle +Lu-tn+t\langle x,\nabla\psi\rangle .\end{equation}
Since $u$ is even, we have that $v_t$ is also even. Since $A$ is symmetric, the restriction $\nu_A$ of $\nu$ onto $A$ is even. It follows that $\int_A\frac{\partial v_t}{\partial x_i}d\nu_A=0$. 
Using \eqref{eq:L-18} and applying the Poincar\'{e} inequality to $\frac{\partial v_t}{\partial x_i}$ with respect to $\nu_A$ and summing over $i=1,\dots,n$, we get
\begin{equation}\label{eq:L-21}\int_A\|\nabla^2u\|^2d\nu\gr \int_A\big(\vartheta_{\nu_A}^{-2}|\nabla v_t|^2+2t\Delta v_t+t^2n\big)\,d\nu.\end{equation}
Substituting \eqref{eq:L-20} into \eqref{eq:L-21} and completing the square we obtain
\begin{align}\label{eq:L-22}\int_A\|\nabla^2u\|^2d\nu &\gr \int_A\big(-t^2\vartheta_{\nu_A}^2|\nabla \psi|^2+2t(Lu-tn+t\langle x,\nabla\psi\rangle )+t^2n\big)\ d\nu\\
\nonumber &=2t\int_ALu\,d\nu-t^2\left(n\,\nu(A)+\int_A\big(\vartheta_{\nu_A}^2|\nabla \psi|^2-2\langle x,\nabla\psi\rangle \big)\,d\nu\right).
\end{align}
The optimal value of $t$ is
$$t=\frac{\int_ALu\,d\nu}{n\,\nu(A)+\int_A\big(\vartheta_{\nu_A}^2|\nabla v_t|^2-2\langle x,\nabla\psi\rangle \big)\,d\nu},$$
which, combined with \eqref{eq:first}, gives
$$\int_K\|\nabla^2u\|^2d\nu\gr \int_A\|\nabla^2u\|^2d\nu\gr \nu(A)\,\frac{\left(\frac{1}{\nu(A)}\int_ALu\,d\nu\right)^2}{n+\frac{1}{\nu(A)}\int_A(\vartheta_{\nu_A}^2|\nabla \psi|^2-2\langle \nabla\psi,x\rangle) \,d\nu}$$
and the theorem follows.
\end{proof}

\begin{remark}\label{rem:1}\rm 
In \cite[Proposition~4.2]{Livshyts-2023}, Livshyts claims a similar bound to Theorem \ref{th:L-reduction} with $\vartheta_{\nu_A}^2$ replaced by $\vartheta_{\nu_K}^2$ but the proof appears to have a gap. More specifically, rather than starting with inequality \eqref{eq:first}, Livshyts pursues all the steps until \eqref{eq:L-22} on the set $K$ instead of $A$ and then bounds the integral on the right hand side of \eqref{eq:L-22} from below by the corresponding integral on $A$. This step however appears to be problematic as the integrand of this expression may not be pointwise nonnegative.
\end{remark}

The next proposition, which originates in \cite[Corollary~4.6]{Livshyts-2023}, shows that for every even isotropic log-concave probability measure $\nu$ on $\mathbb{R}^n$, $n\gr 10$, with
density $f=\exp(-\psi)$, we can find a symmetric convex body $A$ of measure $\nu(A)\gr c$, on which $|\nabla\psi|\ls Cn^2$. We present here a different proof for completeness.

\begin{proposition}\label{prop:large-set}Let $\nu$ be an even isotropic log-concave probability measure on $\mathbb{R}^n$, $n\gr 10$, with
a $C^1$ density $f=\exp(-\psi)$. There exists a symmetric convex set $A\subseteq\mathbb{R}^n$ such that $\nu(A)\gr c_1$ and
$$|\nabla\psi(x)|\ls c_2n^2$$
for all $x\in A$, where $c_1,c_2>0$ are absolute constants.
\end{proposition}

\begin{proof}Consider the set $A=\frac{n-1}{n}R_{3n}(\nu)$. From Lemma~\ref{lem:gamma-ball} we know that $R_{3n}(\nu)\supseteq\frac{1}{3}B_2^n$. It follows that
$$A+\frac{1}{3n}B_2^n=\frac{n-1}{n}R_{3n}(\nu)+\frac{1}{3n}B_2^n\subseteq \frac{n-1}{n}R_{3n}(\nu)+\frac{1}{n}R_{3n}(\nu)=R_{3n}(\nu).$$
Now, let $x\in A$. There exists $v_x\in S^{n-1}$ such that $|\nabla\psi(x)|=\langle\nabla\psi(x),v_x\rangle $. Consider the function
$g(s)=\psi(x+sv_x)$. This is a convex function with $g^{\prime}(0)=\langle\nabla\psi(x),v_x\rangle=|\nabla\psi(x)|$. It follows that
\begin{align*}|\nabla\psi(x)| &=g^{\prime}(0)\ls 3n\big(g\left(\tfrac{1}{3n}\right)-g(0)\big)=3n\big(\psi\left(x+\tfrac{1}{3n}v_x\right)-\psi(x)\big)\\
&\ls 3n\big(\psi\left(x+\tfrac{1}{3n}v_x\right)-\psi (0)\big)
\end{align*}
because $\psi(0)=\min(\psi)$. Since $x\in A$, we have $x+\tfrac{1}{3n}v_x\in R_{3n}(\nu)$, which gives $\psi\left(x+\tfrac{1}{3n}v_x\right)-\psi (0)\ls 3n$. 
It follows that
$$|\nabla\psi(x)|=g^{\prime}(0)\ls 3n\cdot 3n= 9n^2.$$
Finally, using Proposition~\ref{prop:r-2} we see that
$$\nu(A)\gr \left(\frac{n-1}{n}\right)^n\nu(R_{3n}(\mu))\gr \left(\frac{n-1}{n}\right)^n\big(1-e^{-3n/4}\big) \gr c$$
where $c>0$ is an absolute constant.\end{proof}

The main result of Livshyts in \cite{Livshyts-2023} asserts that every even log-concave probability measure on $\mathbb{R}^n$ satisfies the
dimensional Brunn-Minkowski inequality with a constant $c_n\gr c/n^4\ln n$, which is improved in Theorem~\ref{thm:dimBM-main}.
An important new ingredient is provided by the next proposition which combines Proposition~\ref{prop:large-set} with Theorem~\ref{thm:gradient-upper} of Eldan and Klartag, that will be discussed in detail in Section \ref{section:5}.

\begin{proposition}\label{prop:large-set-3}Let $\nu$ be an even isotropic log-concave probability measure on $\mathbb{R}^n$, $n\gr 10$, with
a $C^1$ density $f=\exp(-\psi)$. There exists a symmetric convex set $A\subseteq\mathbb{R}^n$ such that $\nu(A)\gr c_1$ and
$$\int_A|\nabla\psi(x)|^2d\nu(x)\ls Cn^3$$
where $c_1,C>0$ are absolute constants.
\end{proposition}

\begin{proof}Consider the set $A$ from Proposition~\ref{prop:large-set}. We know that $|\nabla\psi(x)|\ls c_2n^2$ for all $x\in A$.
On the other hand, Theorem~\ref{thm:gradient-upper} establishes the bound
$$\int_{\mathbb{R}^n} |\nabla \psi(x)|\, d\mu(x) \ls c_3n,$$
where $c_3>0$ is an absolute constant. Then,
$$\int_A|\nabla\psi(x)|^2d\mu(x)\ls \int_Ac_2n^2|\nabla\psi(x)|\,d\mu(x)\ls c_2n^2\int_{\mathbb{R}^n}|\nabla\psi(x)|\,d\mu(x)\ls Cn^3,$$
where $C=c_2c_3>0$ is an absolute constant.
\end{proof}

We are now in position to complete the proof of the main result of this paper.

\begin{proof}[Proof of Theorem~$\ref{thm:dimBM-main}$] Let $\exp(-\psi)$ be the density of $\mu$. We fix a $\mathcal{C}_+^2$ smooth symmetric convex body $K$ in $\mathbb{R}^n$ and we shall prove that the concavity power satisfies
$$p(\mu, K) \gr  \frac{c}{n^3\ln n}.$$
We will work with the restriction $\mu|_{K}$ of $\mu$ onto $K$, with
density $$\frac{1}{\mu(K)}\mathds{1}_K(x)e^{-\psi(x)}$$ Since $\mu|_{K}$ is centered, there exists an invertible linear map $T$ such that the push-forward $\nu=T_*(\mu|_{K})$ is isotropic. Note that $\nu$ is supported on $TK$ and it is the normalized restriction of $T_\ast\mu$ on $TK$. By Lemma~\ref{lem:posi}, we have
$$p(\mu,K) = p(T_\ast\mu, TK).$$
We write $\exp(-\psi_1)$ for the density of $\nu$ and use Proposition~\ref{prop:large-set} to choose a symmetric convex set 
$A\subseteq TK$ such that $\nu(A)\gr c_1$ and $|\nabla\psi_1(x)|\ls c_2n^2$ for all $x\in A$.

Now, consider a $C^2$ solution $u$ of the equation $Lu\equiv 1$ on $TK$. Since $\nu (TK)=1$, Theorem~\ref{th:L-reduction}
gives 
\begin{align*}
\frac{1}{T_\ast\mu(TK)}\int_{TK} \|\nabla^2u\|^2\, d T_\ast\mu =\int_{TK}\|\nabla^2u\|^2d\nu &\gr \nu(A)\cdot\frac{\left(\frac{1}{\nu(A)}\int_ALu\,d\nu\right)^2}{n+\frac{1}{\nu(A)}\int_A(\vartheta_{\nu_A}^2|\nabla \psi_1|^2-2\langle \nabla\psi_1,x\rangle ) \,d\nu}\\
&=\frac{\nu (A)^2}{n\,\nu(A)+\int_A(\vartheta_{\nu_A}^2|\nabla \psi_1|^2-2\langle \nabla\psi_1,x\rangle )\,d\nu}.
\end{align*}
Note that $\langle \nabla\psi_1,x\rangle \gr 0$ because $\psi_1$ is even and convex. Therefore,
$$n\,\nu(A)+\int_A\big(\vartheta_{\nu_A}^2|\nabla \psi_1|^2-2\langle \nabla\psi_1,x\rangle \big)\,d\nu
\ls n+\int_A\vartheta_{\nu_A}^2|\nabla \psi_1|^2\,d\nu
\ls n+C\vartheta_{\nu_A}^2n^3$$
by Proposition~\ref{prop:large-set-3}. Since $\langle\nabla^2\psi_1\nabla u,\nabla u\rangle \gr 0$, this implies that
$$\int_{TK} \big(\|\nabla^2u\|^2+\langle\nabla^2\psi_1\nabla u,\nabla u\rangle\big)\,d\nu \gr\int_{TK}\big\|\nabla^2u\|^2\,d\nu
\gr \frac{c_1^2}{n+C\vartheta_{\nu_A}^2n^3},$$
and hence \eqref{eq:BM-E-3} gives
$$p(T_\ast\mu,TK) \gr  \frac{1}{C'\vartheta_{\nu_A}^2n^3}$$
for some absolute constant $C'>0$. Since $\nu$ is isotropic and $\nu(A)\gr c_1$, Proposition~\ref{prop:theta-stability-isotropic} implies that $\vartheta_{\nu_A}\ls c_3\sqrt{\ln n}$, and hence
$$p(\mu,K) = p(T_\ast\mu, TK) \gr  \frac{1}{C'\vartheta_{\nu_A}^2n^3}\gr \frac{c}{n^3\ln n}$$
for some absolute constant $c>0$. The conclusion follows from the equivalence of \eqref{eq:BM-E-1} and \eqref{eq:BM-E-2}.
\end{proof}

The discussion in this section leads to the following reduction of the problem to establish a dimensional Brunn-Minkowski inequality with exponent $c_n$.

\begin{proposition}\label{prop:reduction}Let $n\gr 10$. Suppose that there exists a constant $d_n>0$ such that for every even isotropic 
log-concave probability measure $\nu$ on $\mathbb{R}^n$ with a $C^1$ density $f=\exp(-\psi)$ there exists a symmetric convex set 
$A\subseteq\mathbb{R}^n$ such that $\nu(A)\gr c_1$ and
$$|\nabla\psi(x)|\ls d_n$$
for all $x\in A$, where $c_1>0$ is an absolute constant. Then, for any pair of symmetric convex bodies
$K$ and $L$ in $\mathbb{R}^n$ and any $\lambda\in [0,1]$, one has
$$\mu(\lambda K+(1-\lambda)L)^{c_n}\gr\lambda \mu(K)^{c_n}+(1-\lambda)\mu(L)^{c_n},$$
where $c_n\gr c_2/(nd_n\ln n)$ for an absolute constant $c_2>0$.
\end{proposition}

\begin{remark}\label{rem:better}\rm We know that the assumption of Proposition~\ref{prop:reduction} is satisfied with $d_n\approx n^2$. A natural question raised by the proof of
Theorem~\ref{thm:dimBM-main} is whether one can obtain a stronger version of Proposition~\ref{prop:large-set}.

Note that the proof of Theorem~\ref{thm:dimBM-main} uses the convexity of the set $A$ only at the point where it is claimed that
$\vartheta_{\nu_A}\ls C\vartheta_{\nu}$. For this assertion we employ Proposition~\ref{prop:theta-stability-isotropic}, which requires
that $\nu_A$ is a log-concave probability measure, and this forces us to choose our set $A$ to be convex. On the other hand, we know that
$$\int_{\mathbb{R}^n} |\nabla \psi(x)|\, d\nu(x) \ls Cn,$$
hence applying Markov's inequality we see that the set 
$$A_0=\{x:|\nabla\psi(x)|\ls 2Cn\}$$ has measure $\nu(A_0)\gr\frac{1}{2}$, because
$$\nu(\{x:|\nabla\psi(x)|>2Cn\})\ls\frac{1}{2Cn}\int_{\mathbb{R}^n}|\nabla\psi (x)|\,d\nu(x)\ls \frac{1}{2}.$$
Unfortunately, $A_0$ need not be convex as can be seen by the example of the function $\psi:\R^2\to\R$ with
$$\psi(x,y) = \sqrt{x^2+1}+\sqrt{y^2+1}.$$ 
This is clearly an even convex function and
$$\nabla \psi(x,y) = \left(\frac{x}{\sqrt{x^2+1}}, \frac{y}{\sqrt{y^2+1}}\right).$$
Therefore, the set
$$A_0 = \left\{ (x,y)\in\R^2: \ |\nabla\psi(x,y)|\ls 1\right\} = \{(x,y)\in\R^2: \ |x||y| \ls 1\}$$
is evidently non-convex. It is an interesting question whether one can still compare $\vartheta_{\nu_{A_0}}$ and $\vartheta_{\nu}$ in this context. 
Having an estimate $\vartheta_{\nu_{A_0}}\ls C\vartheta_{\nu}$ (or some weaker but good enough estimate) would be enough for a
stronger estimate for $c_n$ in Theorem~\ref{thm:dimBM-main}.
\end{remark}

We conclude this section with a lower bound for the parameter $d_n$ studied in Proposition \ref{prop:reduction}.

\begin{proposition}\label{prop:lower-sup}There exists an even isotropic log-concave probability measure $\mu$ with continuous density $f=e^{-\psi}$ on $\mathbb{R}^n$ such that for every symmetric convex body $B$ in $\R^n$, we have
$$\big\||\nabla \psi|\big\|_{L_\infty(B)}\gr cn,$$
where $c>0$ is an absolute constant.
\end{proposition}

\begin{proof}Let $\nu_K$ be an isotropic measure on $\mathbb{R}^n$ with density $f(x)=e^{-\psi(x)}=\frac{1}{n!\vol_n(K)}e^{-\|x\|_K}$, where $K$ is a $1$-symmetric convex body
and $\|\cdot\|_K$ is the norm induced by $K$. Then, we know that
$$f(x_1,\ldots,x_n)=f(\epsilon_1 x_{\sigma(1)},\ldots,\epsilon_n x_{\sigma(n)})$$ 
for all choices of signs $\epsilon_i\in\{-1,1\}$ and all permutations $\sigma$ of $\{1,\ldots,n\}$. 
Since $K$ is $1$-symmetric, we easily check that $\overline{K}=\vol_n(K)^{-1/n}K$ is isotropic.
Note that 
$$\frac{1}{\left(n!\vol_n(K)\right)^{1/n}}=\|f\|_{\infty}^{1/n}=L_{\nu_K}\approx 1,$$
which implies that
$$\vol_n(K)^{1/n}\approx \frac{1}{(n!)^{1/n}}\approx \frac{1}{n}.$$
It is straightforward to check that $R_t(\nu_K)=tK$ for every $t>0$. We shall show that
\begin{equation}\label{eq:identity}\frac{1}{\nu_K(tK)}\int_{tK}|\nabla \psi(x)|\,d\nu_K(x) =\frac{S(K)}{n\vol_n(K)}\end{equation}
for every $t>0$.

To see this, consider the truncated function $f_t = f \cdot \mathds{1}_{tK} = f\cdot  \mathds{1}_{R_t(\nu_K)}$ which is log-concave and upper semi-continuous. Therefore, the co-area formula \eqref{eq:gen-coarea} yields
\begin{align} \label{equ11}
\int_0^\infty \mathcal{H}^{n-1} (\partial\{x: f_t(x)\gr u\})\, du & = \int_{\R^n} |\nabla f_t(x)| \, d x + \int_{t\partial K} f(x)\, d\mathcal{H}^{n-1}(x)\\ 
\nonumber & = \int_{tK} |\nabla \psi(x)|\, d\nu_K(x) + t^{n-1} e^{-t}f(0)S(K)
\end{align}
as $f|_{t\partial K} \equiv e^{-t}f(0)$. On the other hand,
\begin{align} \label{equ22}
\int_0^\infty \mathcal{H}^{n-1}  (\partial\{x: f_t(x)\gr u\})\, du & = \int_0^\infty e^{-s} f(0) \mathcal{H}^{n-1}(\partial\{x: \ f(x)  \cdot \mathds{1}_{tK}(x) \gr e^{-s} f(0)\})\, ds\\ 
\nonumber & = f(0) \int_0^t e^{-s} \mathcal{H}^{n-1}(s\partial K)\, ds + f(0)\int_t^\infty e^{-s} \mathcal{H}^{n-1}(t\partial K) \, ds\\ 
\nonumber & = \bigg(\int_0^t s^{n-1} e^{-s}\, ds + t^{n-1} e^{-t}\bigg) f(0)S(K).
\end{align}
Combining \eqref{equ11} and \eqref{equ22}, we deduce that
\begin{equation}
\int_{tK} |\nabla \psi(x)|\, d\nu_K(x) = f(0) S(K) \int_0^t s^{n-1} e^{-s} \, ds.
\end{equation}
On the other hand,
\begin{align*}
\nu_K(tK) & = f(0) \int_{tK} e^{-\|x\|_K}\, dx =t^n f(0)\int_Ke^{-t\|y\|_K}dy
=t^n f(0) \int_0^{\infty}e^{-s}\vol_n(\{y\in K:\|y\|_K\ls s/t\})\,ds\\
&=t^n f(0)\left(\vol_n(K)\int_0^t(s/t)^ne^{-s}ds+\vol_n(K)\int_t^{\infty}e^{-s}ds\right) =f(0) \vol_n(K) \left(\int_0^ts^ne^{-s}ds+t^ne^{-t}\right)\\ 
& = n \vol_n(K) f(0) \int_0^t s^{n-1} e^{-s}\, ds,
\end{align*}
where the last equality is because of integration by parts. Formula \eqref{eq:identity} now follows.

Finally, we write 
$$\frac{S(K)}{n\vol_n(K)}=\frac{1}{n\vol_n(K)^{\frac{1}{n}}}\,\frac{S(K)}{\vol_n(K)^{\frac{n-1}{n}}}\approx \frac{S(K)}{\vol_n(K)^{\frac{n-1}{n}}}$$
and choose $K$ to be a multiple of $B_{\infty}^n$. For this choice, the last quantity is of order $n$. Thus,
$$\frac{1}{\nu_K(tK)}\int_{tK}|\nabla \psi(x)|\,d\nu_K(x)\gr cn,$$ which implies $\||\nabla\psi(x)|\|_{L_\infty(tK)}\gr cn$ for every $t>0$. As for every symmetric convex body $B$ there exists $t>0$ for which $tK\subseteq B$, the conclusion follows.
\end{proof}

This shows that the best one can hope with the ideas that are exploited in the present work is a dimensional Brunn--Minkowski inequality with 
exponent $c_n\approx 1/n^2\ln n$. Obtaining an exponent $c_n >\!\!\!> c/n^2 \ln n$  would require further 
ideas.

\section{Bounds for the weighted perimeter}\label{section:4}

As mentioned in the introduction, Theorem \ref{th:nu-f} (whose proof will appear in Section \ref{section:6}) implies Corollary \ref{cor:H}, which asserts that for every 
isotropic log-concave function $f$ on $\mathbb{R}^n$, we have
\begin{equation} \label{eq:H-again}
\int_0^\infty \mathcal{H}^{n-1}(\partial \{x: \ f(x)\gr t\}) \, dt \ls Cn, 
\end{equation}
This will allow us to bound the maximal perimeter of the measure $\mu$ with density $f$.

\begin{proof} [Proof of Theorem \ref{thm:weight-per}]Let $f$ be the density of the isotropic log-concave probability measure $\mu$ on $\R^n$, and let $A\subset\mathbb{R}^n$ be an arbitrary convex body. Set
$$K_t=\{x: \ f(x)\gr t\},\qquad t\gr 0.$$
We have
\begin{equation*}\label{eq:max-per-1}
\mu^+(\partial A) = \int_{\partial A} f(x)\,d\mathcal{H}^{n-1}(x).
\end{equation*}
Applying the layer-cake representation on $\partial A$, we get
\begin{align*}\label{eq:max-per-2}
\int_{\partial A} f(x)\,d\mathcal{H}^{n-1}(x)
&=\int_{\partial A} \int_0^{f(x)} dt\,d\mathcal{H}^{n-1}(x) =\int_0^{\infty}\mathcal{H}^{n-1}\left(\partial A\cap \{f\gr t\}\right)\,dt =\int_0^{\infty}\mathcal{H}^{n-1}(\partial A\cap K_t)\,dt.
\end{align*}
Note that for a closed subset $C\subseteq \R^n$ and an arbitrary subset $M\subseteq \R^n$, we have
$$\partial C \cap M \subseteq \partial (C\cap M).$$
Therefore,
$$\mathcal{H}^{n-1}(\partial A\cap K_t) \ls \mathcal{H}^{n-1}(\partial (A\cap K_t))\ls \mathcal{H}^{n-1}(\partial K_t),$$
where the last inequality is the monotonicity of surface area of convex sets with respect to inclusion. Combining the above, we get by \eqref{eq:H-again} that
$$\mu^+(\partial A) \ls \int_0^{\infty} \mathcal{H}^{n-1}(\partial K_t)\,dt\ls Cn.$$
This implies that $\Gamma(\mu) \leq Cn$, which in turn yields $\Gamma_n \leq Cn$. The asymptotically matching lower bound follows from considering $\mu$ to be the isotropic probability measure on a cube (see also \cite[Section~5]{BGHT}).
\end{proof}

\section{Bounds for the functional perimeter}\label{section:5}

Let $\mu $ be a log-concave probability measure on $\mathbb{R}^n$ with density $f=e^{-\psi}$, where $\psi$ is a convex function.
In this section we give a new proof of the fact that if $f$ is also assumed to be isotropic, then we have
\begin{equation}\label{eq:gradient-1}\int_{\mathbb{R}^n}|\nabla \psi(x)|\,d\mu(x)\ls Cn\end{equation}
for an absolute constant $C>0$. Note that in the $1$-dimensional case if $f$ is symmetric, we get from \cite{Fradelizi-1997}(see~\eqref{eq:dim-1}),
$$\int_{\mathbb{R}} |\psi^{\prime}(x)|\,d\mu(x) = 2 \int_0^{\infty} \psi^{\prime}(x) e^{-\psi(x)}\,dx = 2 e^{-\psi(0)}=2f(0) \ls 2.$$

\bigskip 

\noindent \textbf{\S 5.1. Optimal upper bounds.} Our starting point is the fact (see \cite[Equation~4]{Alves-Gonzalez-Villa-2023})
that if $f$ is an integrable log-concave function such that
\begin{equation}\label{eq:condition}f\gr \alpha\mathds{1}_{B_2^n}\end{equation}
for some $\alpha>0$ then
\begin{equation}\label{eq:very-useful}\int_{\mathbb{R}^n}|\nabla f(x)|\,dx\ls n\int_{\mathbb{R}^n}f(y)\,dy+\int_{\mathbb{R}^n}f(z)\ln\left(\frac{f(z)}{\alpha\|f\|_{\infty}}\right)\,dz.\end{equation}
This follows from \cite[Lemma~4.3]{Alonso-Gonzalez-Jimenez-Villa-2018}.

Since we need a variant of this inequality, we present the full details. In what follows, the sum of two log-concave
functions is given by the sup-convolution (or Asplund product)
$$(f\star g)(x)=\sup_{y\in\mathbb{R}^{n}}\big(f(y)g(x-y)\big).$$
The dilation operation is given by $(t\cdot f)(x)=f\left(\frac{x}{t}\right)^t$.

\begin{lemma}\label{lem:alonso}Let $f:\mathbb{R}^n\to [0,\infty)$ be a log-concave function. For any $r,a>0$ we have
$$\lim_{t\to 0^+}(f\star t\cdot (a\mathds{1}_{rB_2^n}))(z)=f(z)$$
at every point of continuity $z$ of $f$, and
$$\lim_{t\to 0^+}\frac{(f\star t\cdot (a\mathds{1}_{rB_2^n}))(z)-f(z)}{t}=r|\nabla f(z)|+f(z)\ln a$$
almost everywhere.\end{lemma}

\begin{proof}By the definition of the Asplund product we have
\begin{align*}&\lim_{t\to 0}(f\star t\cdot (a\mathds{1}_{rB_2^n}))(z)=\lim_{t\to 0}\sup\left\{f(x)\cdot \mathds{1}_{rB_2^n}(y/t)a^t:z=x+y\right\} \\
&\hspace*{1cm}=\lim_{t\to 0}\sup\left\{f(z-rty)a^t:y\in B_2^n\right\}=f(z)
\end{align*}
if we assume that $f$ is continuous at $z$. For the second assertion of the lemma, we start by writing
\begin{align*}
\lim_{t\to 0^+}\frac{(f\star t\cdot (a\mathds{1}_{rB_2^n}))(z)-f(z)}{t} &=\lim_{t\to 0^+}\,\sup_{y\in B_2^n}\frac{f(z-rty)a^t-f(z)a^t+f(z)a^t-f(z)}{t}\\
&=\lim_{t\to 0^+}\,\sup_{y\in B_2^n}a^t\frac{f(z-rty)-f(z)}{t}+\lim_{t\to 0^+}f(z)\frac{a^t-1}{t}.
\end{align*}
As the function $\psi=-\ln f$ is convex, it has a first order Taylor expansion around almost every point $z\in \{f>0\}=\{\psi<\infty\}$ (and even a second order Taylor expansion, by Alexandrov's theorem) and thus the same holds for $f$. Moreover, for such points $z$ the gradient $\nabla f(z)$ is uniquely defined and we have
$$f(z+w) = f(z) + \langle \nabla f(z), w\rangle + c_z(|w|)$$
for some function $c_z:(0,\infty)\to \R$ with $\lim_{t\to0^+} c_z(t)/t =0$. Therefore,
\begin{align*}
\lim_{t\to 0^+}\,\sup_{y\in B_2^n}a^t\frac{f(z-rty)-f(z)}{t} & = \lim_{t\to 0^+}\, a^t \, \sup_{y\in B_2^n} \left\{- r \langle \nabla f(z), y\rangle + \frac{c_z(rt)}{t} \right\}\\ & = \lim_{t\to0^+} a^t \Big( r|\nabla f(z)| + \frac{c_z(rt)}{t} \Big) = r|\nabla f(z)|.
\end{align*}
Evidently the same holds for every $z\in\R^n \setminus \overline{\{f>0\}}$ and thus it holds for almost every $z\in\R^n$ as the boundary of $\{f>0\}$ always has Lebesgue measure 0. Finally, $\lim_{t\to 0^+}\frac{a^t-1}{t}=\ln a$, and the lemma follows.
\end{proof}

The next lemma is a variant of \cite[Lemma~4.3]{Alonso-Gonzalez-Jimenez-Villa-2018}. Similar computations can be found in \cite{Colesanti-Fragala-2013}.

\begin{lemma}\label{lem:col-frag}Let $f=e^{-\psi}:\mathbb{R}^n\to [0,\infty)$ be an integrable log-concave function. Then,
$$\lim_{t\to 0^+}\frac{\int_{\mathbb{R}^n}(f\star (t\cdot f))(x)\,dx-\int_{\mathbb{R}^n}f(x)\,dx}{t}=n\int_{\mathbb{R}^n}f(x)\,dx +
\int_{\mathbb{R}^n}f(x)\ln f(x)\,dx.$$
\end{lemma}

\begin{proof}First we observe that
$$(f\star (t\cdot f))(z)=e^{-(1+t)\psi\left(\frac{z}{1+t}\right)}$$
for all $z\in\mathbb{R}^n$. To see this, note that if $z=x+y$ then
\begin{align*}\psi(x)+t\psi(y/t) &=(1+t)\left(\frac{1}{1+t}\psi(x)+\frac{t}{1+t}\psi(y/t)\right)\gr (1+t)\psi\left(\frac{1}{1+t}x+\frac{t}{1+t}\frac{y}{t}\right)\\
&=(1+t)\psi\left(\frac{z}{1+t}\right)
\end{align*}
with equality if $x=y/t=z/(1+t)$. Now, we can write
\begin{align*}
\frac{\int_{\mathbb{R}^n}(f\star (t\cdot f))(x)\,dx-\int_{\mathbb{R}^n}f(x)\,dx}{t} &=
\frac{1}{t}\left((1+t)^n\int_{\mathbb{R}^n}e^{-(1+t)\psi(x)}dx-\int_{\mathbb{R}^n}e^{-\psi(x)}dx\right)\\
&=\frac{(1+t)^n-1}{t}\int_{\mathbb{R}^n}e^{-(1+t)\psi(x)}dx+\int_{\mathbb{R}^n}e^{-\psi(x)}\,\frac{e^{-t\psi(x)}-1}{t}\,dx.
\end{align*}
Since $\lim_{t\to 0^+}\frac{1}{t}((1+t)^n-1)=n$ and $\lim_{t\to 0^+}\frac{e^{-t\psi(x)}-1}{t}=-\psi(x)=\ln f(x)$, applying
the monotone convergence theorem we conclude the proof.
\end{proof}

We are now in position to prove Theorem \ref{thm:gradient-upper}.

\begin{proof}[Proof of Theorem~$\ref{thm:gradient-upper}$] From Lemma~\ref{lem:gamma-ball} we know that if $n\gr 10$ then
$$R_{3n}(\mu)\supseteq  \tfrac{1}{3}B_2^n.$$
By the definition of $R_{3n}(\mu)$ it follows that
\begin{equation}\label{eq:main-10}e^{-3n}f(0)\mathds{1}_{\tfrac{1}{3}B_2^n}(x)\ls f(x), \qquad x\in\mathbb{R}^n.\end{equation}
and hence \eqref{eq:condition} is satisfied with $\alpha =e^{-3n}f(0)$ but with radius $\tfrac13$ instead of 1. Then, Lemma~\ref{lem:col-frag} shows that
\begin{equation}\label{eq:main-11}\lim_{t\to 0^+}\frac{\int_{\mathbb{R}^n}(f\star (t\cdot f))(x)\,dx-\int_{\mathbb{R}^n}f(x)\,dx}{t}=n\int_{\mathbb{R}^n}f(x)\,dx
+\int_{\mathbb{R}^n}f(x)\ln f(x)\,dx.\end{equation}
On the other hand, Lemma~\ref{lem:alonso} shows that 
\begin{equation}\label{eq:main-12}\lim_{t\to 0^+}\frac{(f\star t\cdot (e^{-3n}f(0)\mathds{1}_{\tfrac{1}{3}B_2^n}))(z)-f(z)}{t}=
\frac{1}{3}|\nabla f(z)|+f(z)\ln (e^{-3n}f(0))\end{equation}
almost everywhere. Note that
$$(f\star t\cdot (a\mathds{1}_{rB_2^n}))(z)=\sup\left\{f(x)\cdot \mathds{1}_{rB_2^n}(y/t)a^t:z=x+y\right\}=\sup\left\{f(z-rty)a^t:y\in B_2^n\right\}\gr a^tf(z)$$
for every $a,r>0$. Therefore,
\begin{align*}h_t(z):=\frac{(f\star t\cdot (e^{-3n}f(0)\mathds{1}_{\tfrac{1}{3}B_2^n}))(z)-f(z)}{t}&\gr f(z) \frac{(e^{-3n}f(0))^t-1}{t}\\
&\gr\ln (e^{-3n}f(0))\,f(z).
\end{align*}
This last function is integrable, and hence applying Fatou's lemma to the family of nonnegative functions $\{h_t(z) - \ln(e^{-3n}f(0)) f(z)\}_{t>0}$ and combining \eqref{eq:main-10}--\eqref{eq:main-12} we get
\begin{equation}\label{eq:main-13}
\frac{1}{3}\int_{\mathbb{R}^n}|\nabla f(z)|\,dz+\int_{\mathbb{R}^n}f(z)\ln (e^{-3n}f(0))\,dz\ls n\int_{\mathbb{R}^n}f(x)\,dx+
\int_{\mathbb{R}^n}f(x)\ln f(x)\,dx.\end{equation}
Note that $\int_{\mathbb{R}^n}f(z)\,dz=1$ and $\nabla f(x)=-f(x)\nabla\psi(x)$. Moreover, since $f$ is a centered log-concave function, from Jensen's inequality
we have
\begin{equation}\label{eq:fradelizi-e^n-1}\ln f(0)=\ln f\left(\int_{\mathbb{R}^n} xf(x)dx\right)\gr\int_{\mathbb{R}^n}
f(x)\ln f(x)dx.\end{equation}
Therefore, \eqref{eq:main-13} yields
\begin{align}\label{eq:main-14}\frac{1}{3} \int_{\mathbb{R}^n}|\nabla\psi(x)|\,d\mu(x) &\ls \ln(e^{3n}f(0)^{-1})+n+\int_{\mathbb{R}^n}f(x)\ln f(x)\,dx\\
\nonumber &\ls 4n-\ln f(0)+\int_{\mathbb{R}^n}f(x)\ln f(x)\,dx\ls 4n.\end{align}
This completes the proof of the theorem.
\end{proof}

\begin{remark} \rm
Theorem \ref{thm:gradient-upper} contains a sharp upper bound for the first moment of $\nabla \psi$ with respect to the isotropic log-concave probability measure $\mu$ with density $e^{-\psi}$. This is the best integrability that one can hope for in this setting, even in the one-dimensional case. To see this, fix $p\gr 1$ and consider the even log-concave function $f_p:\R\to\R$ given by
$$
f_p(x)
=
\frac{p}{2}
\frac{\sqrt{\Gamma\!\left(3/p\right)}}{\Gamma\!\left(1/p\right)^{3/2}}
\exp\!\left(
-
\left(\frac{\Gamma\!\left(3/p\right)}{\Gamma\!\left(1/p\right)}\right)^{p/2}
|x|^p
\right)
$$
which satisfies $\int_\R f_p(x)\, dx = 1 = \int_\R x^2 f_p(x) \ dx$ and is thus an isotropic probability density. Moreover,
$$
\int_\R |(-\log f_p)'(x)|^{1+\alpha} f_p(x) \, dx = \int_{\mathbb{R}} \frac{|f_p'(x)|^{1+\alpha}}{f_p(x)^\alpha}\,dx
=
p^{1+\alpha} \, \left(\frac{\Gamma(3/p)}{\Gamma(1/p)} \right)^{\frac{1+\alpha}{2}} \frac{\Gamma\left(\frac{(p-1)(1+\alpha)+1}{p}\right)}{\Gamma(1/p)}
$$
for every $\alpha>0$. Using that $\Gamma(\varepsilon) \sim 1/\varepsilon$ as $\e\to0^+$, we obtain
$$\int_{\mathbb{R}} \frac{|f_p'(x)|^{1+\alpha}}{f_p(x)^\alpha}\,dx \sim p^{1+\alpha} \, \left(\frac{p/3}{p} \right)^{\frac{1+\alpha}{2}} \frac{\Gamma(1+\alpha)}{p} = 3^{-\frac{1+\alpha}{2}} \Gamma(1+\alpha) p^\alpha$$
which is unbounded as $p\to\infty$.
\end{remark}

\begin{remark}\rm A result similar to Theorem \ref{thm:gradient-upper} for log-concave functions in John's position, rather than in isotropic position, 
appears in \cite[Remark~3.3]{GK25}.
\end{remark}

Considering the special case where $\mu$ has a radial density, we see that there exist isotropic log-concave probability measures for
which we have a $O(\sqrt{n})$ bound.

\begin{proposition}\label{prop:radial}
Let $\mu$ be a radial isotropic log-concave probability measure on $\mathbb{R}^n$, where $n\gr2$. Assume that the density $f$ of $\mu$ is of the form
$f=e^{-\psi}$, where $\psi(x)=g(|x|)$ for a continuously differentiable function $g:[0,\infty)\to\mathbb{R}$.
Then,
$$\int_{\mathbb{R}^n}|\nabla\psi(x)|\,d\mu(x) \ls \sqrt{n+1}.$$
\end{proposition}

\begin{proof}We shall use the following result of Borell (see \cite[Theorem~2.2.5]{BGVV-book}). If $G=e^{-g}:[0,\infty)\to [0,\infty)$ is a log-concave function, 
then the function
$$\Psi_g(p)=\frac{\int_0^{\infty}r^pe^{-g(r)}\,dr}{\Gamma(p+1)}$$
is log-concave on $[0,\infty)$. 
Note that
$$\int_{\mathbb{R}^n}|x|^qd\mu(x)=n\omega_n\int_0^{\infty}r^{q+n-1}e^{-g(r)}dr = n\omega_n \Gamma(q+n) \Psi_g(q+n-1)$$
for every $q>-(n-1)$. Since $\mu$ is a probability measure we have 
$$n\omega_n\Gamma(n)\Psi_g(n-1)=n\omega_n\int_0^{\infty}r^{n-1}e^{-g(r)}dr=\mu(\mathbb{R}^n)=1,$$ 
and since $\mu$ is isotropic we also have 
$$n\omega_n\Gamma(n+2)\Psi_g(n+1)=n\omega_n\int_0^{\infty}r^{n+1}e^{-g(r)}dr=\int_{\mathbb{R}^n}|x|^2d\mu(x)=n.$$
Recall that $f=e^{-\psi}$ with $\psi(x)=g(|x|)$. Therefore, integration by parts shows that
\begin{align*}
I:=\int_{\mathbb{R}^n}|\nabla \psi(x)|\,e^{-\psi(x)}dx &=n\omega_n\int_0^{\infty}r^{n-1}g^{\prime}(r)e^{-g(r)}dr\\
&=(n-1)\,n\omega_n\int_0^{\infty}r^{n-2}e^{-g(r)}dr\\
&=(n-1)n\omega_n\Gamma(n-1)\Psi_g(n-2).
\end{align*}
Writing $n-1=\frac{2}{3}(n-2)+\frac{1}{3}(n+1)$ and using the log-concavity of $\Psi_g$ we get
$$\Psi_g(n-1)^3\gr \Psi_g(n-2)^2\Psi_g(n+1)$$
or equivalently,
$$\frac{1}{\Gamma(n)^3}\gr\frac{I^2}{(n-1)^2\Gamma(n-1)^2}\,\frac{n}{\Gamma(n+2)}.$$
It follows that
$$I^2\ls \left(\frac{(n-2)!}{(n-1)!}\right)^2\frac{(n+1)!}{(n-1)!}\,\frac{(n-1)^2}{n}
=\frac{1}{(n-1)^2}\,n(n+1)\,\frac{(n-1)^2}{n}=n+1,$$ 
which proves the proposition. \end{proof}

\bigskip 

\noindent \textbf{\S 5.2. Optimal lower bounds.} The next proposition shows that the quantity $\int_{\mathbb{R}^n}|\nabla\psi(x)|\,d\mu(x)$ is always bounded from below by $c\sqrt{n}$.

\begin{proposition}\label{th:root-of-n}
Let $\mu$ be an isotropic log-concave probability measure on $\mathbb{R}^n$ with essentially continuous density $f=e^{-\psi}$. 
Then,
$$\int_{\mathbb{R}^n} |\nabla \psi(x)|\,d\mu(x) \gr c\sqrt{n},$$
where $c>0$ is an absolute constant.
\end{proposition}

\begin{proof}
Since $f$ is essentially continuous, it is also $W^{1,1}$ (see \cite[Remark~3.3]{Rotem-2022}) and thus, the classical Sobolev inequality with a sharp constant yields
$$\int_{\mathbb{R}^n} |\nabla \psi(x)|\,d\mu(x) = \int_{\R^n} |\nabla f|\, dx \geq n\omega_n^{1/n} \bigg(\int_{\R^n} f^{\frac{n}{n-1}}\, dx \bigg)^{\frac{n-1}{n}} \geq n\omega_n^{1/n} \int_{\R^n} f(x)\, dx \approx \sqrt{n},$$
where the second inequality follows from Jensen's inequality.
\end{proof}

The next proposition shows the sharpness of Theorem \ref{thm:gradient-upper}.

\begin{proposition}\label{prop:lower-n}There exists an even isotropic log-concave probability measure $\mu$ on $\mathbb{R}^n$ such that
$$\int_{\mathbb{R}^n}|\nabla \psi(x)|\,d\mu(x)\gr cn,$$
where $c>0$ is an absolute constant.
\end{proposition}

\begin{proof}Let $\nu_K$ be an isotropic measure on $\mathbb{R}^n$ with density $f(x)=e^{-\psi(x)}=\frac{1}{n!\vol_n(K)}e^{-\|x\|_K}$, where $K$ is a $1$-symmetric convex body
and $\|\cdot\|_K$ is the norm induced by $K$. Since $K$ is $1$-symmetric,  $\overline{K}=\vol_n(K)^{-1/n}K$ is isotropic. Recall from the proof of Proposition \ref{prop:lower-sup} that $\vol_n(K)^{1/n}\approx \frac{1}{n}$ and that $R_t(\nu_K)=tK$ for every $t>0$. Applying \eqref{eq:identity} with $t\to\infty$, we get
\begin{align*}
\int_{\mathbb{R}^n}|\nabla \psi(x)|\,d\nu_K(x) 
=\frac{S(K)}{n\vol_n(K)} =\frac{1}{n\vol_n(K)^{\frac{1}{n}}}\,\frac{S(K)}{n\vol_n(K)^{\frac{n-1}{n}}}\approx \frac{S(K)}{\vol_n(K)^{\frac{n-1}{n}}}.
\end{align*}
It remains to observe that
$$\frac{S(K)}{\vol_n(K)^{\frac{n-1}{n}}}=\frac{S(\overline{K})}{\vol_n(\overline{K})^{\frac{n-1}{n}}}\approx n$$
if we choose $K$ so that $\overline{K}=\frac{1}{2}B_{\infty}^n$.
\end{proof}

\section{Reformulations and geometric applications}\label{section:6}

In this section we collect a number of reformulations of the gradient estimate from Theorem~\ref{thm:gradient-upper}. Some of the results are direct applications
of this estimate, once the relevant language has been established.

\bigskip 

\noindent \textbf{\S 6.1. Moment measure of a log-concave function.} Let $f=e^{-\psi}:\mathbb{R}^n\to [0,\infty)$ be 
a pointwise finite log-concave function with finite positive integral. We define the moment measure $\mu_f$ of $f$ to be the Borel measure 
on $\mathbb{R}^n$ which is the push-forward of $fdx$ under $\nabla\psi$. Equivalently,
\begin{equation}\label{eq:defimage}
\int_{\mathbb{R}^n} g(y) \, d\mu_f(y) = \int_{\mathbb{R}^n} g(\nabla \psi(x)) f(x)\, dx
\end{equation}
for every Borel measurable function $g$ such that $g\in L^1(\mu_f)$ or $g$ is non-negative.

The question that was addressed by Cordero-Erausquin and Klartag in \cite{Cordero-Klartag-2015} is to characterize those measures $\nu$ that
are moment measures of log-concave functions $f$ on $\mathbb{R}^n$ with finite positive integral. Under the assumption that $f$ is essentially continuous
(or, equivalently, that $\psi$ is essentially continuous) they showed that the moment measure $\mu_f$ of a log-concave function $f$ as above has the
following two properties: it is centered and it is not supported by a hyperplane. In particular, they showed \mbox{that the first moment of $\mu_f$ is finite: }
\begin{equation}\label{eq:defimage-finite}
\int_{\mathbb{R}^n} |y| \, d\mu_f(y) = \int_{\mathbb{R}^n} |\nabla \psi(x)|\,f(x) \, dx<+\infty.
\end{equation}
It turns out that these necessary conditions are also sufficient. The main result of \cite{Cordero-Klartag-2015} shows that there is a bijection
between essentially-continuous, convex functions $\psi $ modulo translations, and finite measures on $\mathbb{R}^n$ that are centered and not supported by a hyperplane.

\begin{theorem}[Cordero-Erausquin--Klartag]\label{th:Cordero-Klartag}Let $\nu$ be a Borel measure on $\mathbb{R}^n$ such that
\begin{enumerate}
\item[{\rm (i)}] $0 < \nu(\mathbb{R}^n) < +\infty$.
\item[{\rm (ii)}] The measure $\nu$ is not supported by a lower-dimensional subspace.
\item[{\rm (iii)}] The barycenter of $\nu$ lies at the origin; in particular, $\nu$ has finite first moment.
\end{enumerate}
Then, there exists an essentially continuous convex function $\psi: \mathbb{R}^n \rightarrow \mathbb{R}\cup \{ + \infty \}$
such that $\nu$ is the moment measure of the log-concave function $f=e^{-\psi}$.
Moreover, this function $\psi$ is uniquely determined up to translation.
\end{theorem}

Assuming that $f$ is an isotropic log-concave function on $\mathbb{R}^n$, Theorem~\ref{thm:gradient-upper} immediately implies a quantitative version
of \eqref{eq:defimage-finite}.

\begin{theorem}\label{th:moment-mu-f}Let $\mu_f$ be the moment measure of an isotropic log-concave function $f$ on $\mathbb{R}^n$.
Then,
$$\int_{\mathbb{R}^n} |y| \, d\mu_f(y)\ls Cn$$
for some absolute constant $C>0$.
\end{theorem}

\bigskip 

\noindent \textbf{\S 6.2. Surface area measures of a log-concave function.} Recall that for every convex body $K$ in $\mathbb{R}^n$ there exists 
a Borel measure $S_K$ on $S^{n-1}$, the surface area measure of $K$, such that
\begin{equation}\label{eq:surface-area-measure}\lim_{t\to 0^+}\frac{\vol_n(K+tL)-\vol_n(K)}{t}=\int_{S^{n-1}}h_L(\xi)\,dS_K(\xi)\end{equation}
for every convex body $L$ in $\mathbb{R}^n$. Choosing $L=B_2^n$ we see that $S_K(S^{n-1})$ equals the surface area of $K$.

A natural question in ``functional convexity" is to extend the notion of the surface area measure of a convex body to the setting
of functions in ${\rm LC}_n$. For any log-concave function $f:\mathbb{R}^{n}\to\mathbb{R}$ with $0<\int f<\infty$, Rotem considered 
in \cite{Rotem-2022} and \cite{Rotem-2023} two surface area measures of $f$. Write $f=e^{-\psi}$, where $\psi:\mathbb{R}^{n}\to(-\infty,\infty]$ 
is a convex function. Then, the first surface area measure $\mu_{f}$ of $f$ is a measure on $\mathbb{R}^{n}$, defined as the push-forward
$$\mu_{f}=(\nabla\psi)_{\ast}(fdx).$$
Note that, in the terminology of \S 6.1, $\mu_f$ is precisely the moment measure of $f$. The second surface area measure of $f$ is a measure $\nu_{f}$
on the sphere $S^{n-1}$, defined as the push-forward 
$$\nu_{f}=(n_{K_{f}})_{\ast}(fd\mathcal{H}^{n-1}|_{\partial K_{f}}),$$
where $K_{f}=\overline{\{ x:\ f(x)>0\}}$ is the support of $f$, and $n_{K_{f}}$ denotes the Gauss map $n_{K_{f}}:\partial K_{f}\to S^{n-1}$. 

For example, if $f(x)=e^{-|x|^{2}/2}$ then $\mu_{f}=e^{-|x|^{2}/2}dx$ and $\nu_{f}\equiv 0$, because $\partial K_{f}=\partial\mathbb{R}^{n}=\varnothing$.
On the other hand, if $f=\mathds{1}_K$ for some convex body $K$ in $\mathbb{R}^n$ then $\mu_{\mathds{1}_{K}}=\vol_n(K)\,\delta_{0}$ 
and $\nu_{\mathds{1}_{K}}=S_{K}$, the usual surface area measure. The pair $(\mu_{f},\nu_{f})$ is the pair of
surface area measures of $f$. 

It is useful to observe that we do not assume any regularity for the definitions of $\mu_{f}$ and $\nu_{f}$. Since $\psi=-\ln f$
is a convex function, it is differentiable almost everywhere with respect to Lebesgue measure on the set $K_f=\overline{\{ x:\ \psi(x)<\infty\}}$. 
Therefore, the push-forward $(\nabla\psi)_{\ast}(fdx)$ is well-defined. Similarly, since $K_{f}$ is a closed convex set, its
boundary $\partial K_{f}$ is a Lipschitz manifold, and hence the Gauss map $n_{K_{f}}$ is defined $\mathcal{H}^{n-1}$-almost everywhere
and the push-forward is again well-defined. 

Motivated by \eqref{eq:surface-area-measure}, for any $f,g\in{\rm LC}_{n}$ one may consider the first variation of the integral of $f$
in the direction of $g$, defined by 
\begin{equation}\delta(f,g)=\lim_{t\to0^{+}}\frac{\int f\star\left(t\cdot g\right)-\int f}{t}.\label{eq:variation-def}
\end{equation}
Under some additional regularity assumptions, $\delta(f,g)$ was studied by Colesanti and Fragal\`{a} in \cite{Colesanti-Fragala-2013}.
The same first variation appears in the work of Cordero-Erausquin and Klartag \cite{Cordero-Klartag-2015} on moment measures that we
briefly discussed in \S 6.1, as a step in their proof of the Minkowski-type existence Theorem~\ref{th:Cordero-Klartag}. Rotem dropped
the additional hypotheses in \cite{Rotem-2022} and proved that if $f=e^{-\psi},g=e^{-\varphi}\in {\rm LC}_n$, $0<\int f<\infty$ and $\nu_{f}=0$ then
\begin{equation}\label{eq:delta}\delta(f,g)=\int_{\mathbb{R}^{n}}h_{g}d\mu_{f}\end{equation}
with no regularity assumptions, where
$$h_g=\mathcal{L}\varphi$$
is the Legendre transform of $\varphi$, which is the functional analogue of the support function of $g$. Note that since $f$ is log-concave
and upper semi-continuous, it is only discontinuous at points $x\in\partial K_{f}$
such that $f(x)\ne0$. Therefore the condition $\nu_{f}=0$ is equivalent
to the statement that $f$ is essentially continuous. In his subsequent work \cite{Rotem-2023}, Rotem obtained a very general version of \eqref{eq:delta}.
 
\begin{theorem}[Rotem]\label{th:Rotem}Let $f,g\in {\rm LC}_{n}$ such that $0<\int f<\infty$. Then, 
\begin{equation}
\delta(f,g)=\int_{\mathbb{R}^{n}}h_{g}d\mu_{f}+\int_{S^{n-1}}h_{K_{g}}d\nu_{f}.\label{eq:Rotem}
\end{equation}
\end{theorem}

Rotem proved in \cite[Proposition~1.6]{Rotem-2023} that $\nu_f(S^{n-1})<+\infty$. Assuming that $f$ is an isotropic log-concave function on $\mathbb{R}^n$, 
we shall complement the estimate of Theorem~\ref{th:moment-mu-f} by providing a similar estimate for $\nu_f(S^{n-1})$. More precisely, we will show that
$$\int_{\R^n} |y|\, d\mu_f(y) + \nu_f(S^{n-1}) \leq Cn .$$
This is the content of Theorem \ref{th:nu-f} and it is again a consequence of Theorem~\ref{thm:gradient-upper}.

\begin{proof}[Proof of Theorem~$\ref{th:nu-f}$]
Write $f=e^{-\psi}$ for a proper lower semi-continuous convex function $\psi:\R^n\to(-\infty,\infty]$. For $\lambda>0$,
let $\psi_\lambda:\R^n\to\R$ be the infimum convolution of $\psi$ with the function $\frac{1}{2\lambda}|x|^2$, namely
$$\psi_\lambda(x) = \inf_{y\in\R^n} \Big\{ \psi(y) + \frac{1}{2\lambda} |x-y|^2\Big\}.$$
The function  $\psi_\lambda$ is also referred to as the Moreau envelope of $\psi$, see \cite[Definition~1.22]{RW98}. Notice that $\psi_\lambda$ is always finite and $\psi_\lambda(x) \ls \psi(x)$ at points $x$ where $\psi(x)<\infty$. Observe further that the function
$$\Theta_\lambda (x,y) = \psi(y) + \frac{1}{2\lambda} |x-y|^2$$
is jointly convex in $(x,y)$ and thus $\psi_\lambda = \inf_y \Theta_\lambda(\cdot,y)$ is also convex. Moreover, it is a classical fact (see, e.g., \cite[Theorem~2.26 (b)]{RW98}) that $\psi_\lambda$ is of class $C^1$ on $\R^n$.

Recall that a sequence of convex lower semi-continuous functions $\phi_n:\R^n\to(-\infty,\infty]$
epi-converges to $\phi:\R^n\to(-\infty,\infty]$ if for every $x\in\R^n$ the following conditions are met:

\smallskip

\noindent (i) for every sequence $x_n\to x$, we have
$$\liminf_{n\to\infty} \phi_n(x_n) \gr \phi(x)$$
\noindent (ii) there exists a sequence $x_n\to x$ for which
$$\limsup_{n\to\infty} \phi_n(x_n) \ls \phi(x).$$

Given a sequence $\lambda_n \searrow 0$ as $n\to\infty$ we will show that the sequence of functions $\{\psi_{\lambda_n}\}_{n\gr1}$ epi-converges to $\psi$. It is well known that $\psi_{\lambda_n}$ converges pointwise to $\psi$ as $n\to\infty$ (see, e.g., \cite[Theorem~1.25]{RW98}) and thus property (ii) holds trivially for $x_n=x$. To verify property (i), fix some $x\in\R^n$ and a sequence $x_n\to x$. By the Fenchel--Moreau theorem, since $\psi$ is proper, lower semi-continuous and convex, we have 
\begin{equation} \label{eq:fm}
\psi(x) = \sup \big\{ \ell(x): \ \ell \ \mbox{is an affine function with } \ell \ls \psi\big\}
\end{equation}
for every $x\in\R^n$. Therefore, for such an affine function $\ell(y) = \langle y,a\rangle + b$, we also have
\begin{equation*}
\begin{split}
\psi_{\lambda_n}(x_n) \gr \inf_{y\in\R^n} \Big\{ \ell(y) + \frac{1}{2\lambda_n} |x_n-y|^2 \Big\} = \inf_{y\in\R^n} \Big\{ \langle y,a\rangle+b + \frac{1}{2\lambda_n} |x_n-y|^2 \Big\}
= \langle x_n,a\rangle + b - \frac{\lambda_n}{2} |a|^2
\end{split}
\end{equation*}
as the infimum is attained at $y=x_n - \lambda_n a$. Therefore,
$$\liminf_{n\to\infty} \psi_{\lambda_n}(x_n) \gr \langle x,a\rangle+b = \ell(x)$$
and condition (i) follows from equation \eqref{eq:fm}. This establishes that $\psi_\lambda$ epi-converges to $\psi$ as $\lambda\to0^+$.

Now consider $f_\lambda = e^{-\psi_\lambda}:\R^n \to(0,\infty)$ and observe that $f_\lambda$ is a $C^1$ log-concave function. Since $f$ is integrable, there exist constants $c_1, c_2 >0$ such that
$$\psi(x) \gr c_1 |x| - c_1$$
for every $x\in\R^n$. Therefore, we also have
\begin{equation*}
\begin{split}
\psi_\lambda(x)  \gr \inf_{y\in\R^n} \Big\{ c_1 |y| -c_2 + \frac{1}{2\lambda} |x-y|^2 \Big\} & \gr \inf_{y\in\R^n} \Big\{ c_1|y|-c_2 + \frac{1}{2\lambda}\big( |x|^2 - 2 |x||y| + |y|^2 \big)\Big\} \\ 
& = \inf_{t\gr0} \Big\{ \frac{1}{2\lambda}|x|^2 - c_2 + \Big(c_1-\frac{|x|}{\lambda}\Big)t + \frac{t^2}{2\lambda}\Big\}.
\end{split}
\end{equation*}
The infimum of the latter expression is attained at $t_\ast = \lambda c_1-|x|$. If $t_\ast\gr0$, this gives
$$\psi_\lambda(x) \gr c_1|x| - c_2 - \frac{\lambda c_1^2}{2}$$
whereas if $t_\ast<0$ then plugging $t=0$ we get
$$\psi_\lambda(x) \gr \frac{1}{2\lambda}|x|^2 - c_2 \gr c_1|x| - c_2 -\frac{\lambda c_1^2}{2},$$
since the second inequality holds for every $x\in \R^n$. This way we conclude that
$$f_\lambda(x) \ls \exp\big(-c_1|x|+c_2+\tfrac{\lambda c_1^2}{2}\big)$$
for every $x\in \R^n$. In particular, the integrands of the parameters
$$Z_\lambda = \int_{\R^n} f_\lambda(x)\, dx , \ \ \ m_\lambda = \int_{\R^n} x \, f_\lambda(x)\, dx \ \ \ \mbox{and} \ \ \ \Sigma_\lambda = \int_{\R^n} \big(x - m_\lambda) \otimes (x-m_\lambda)\, f_\lambda(x)\, dx$$
are all bounded above by uniformly integrable functions when $\lambda\in(0,1)$. Since $f_\lambda$ converges pointwise to $f$, the dominated convergence theorem yields
$$Z_\lambda \longrightarrow \int_{\R^n} f(x)\, dx = 1, \ \ \ m_\lambda \longrightarrow \int_{\R^n} x\, f(x)\, dx = \operatorname{bar}(f) = 0$$
and
$$\Sigma_\lambda \longrightarrow \int_{\R^n} x\otimes x \, f(x)\, dx = \operatorname{Cov}(f)= I_n$$
as $\lambda\to0^+$. In other words, the renormalized $C^1$ log-concave functions
$$F_\lambda(x) = \frac{|\det \Sigma_\lambda|^{1/2}}{Z_\lambda} f_\lambda\big(\Sigma_\lambda^{1/2}x+m_\lambda\big)$$
are isotropic and converge to $f$ in the topology of epi-convergence of their convex potentials.

Write $F_\lambda = e^{-\Psi_\lambda}$ for some $C^1$ convex function $\Psi_\lambda:\R^n\to\R$. Invoking a recent result of Falah and Rotem \cite[Theorem~1.9]{FR26}, we deduce that the pair of surface area measures $(\mu_{F_\lambda}, \nu_{F_\lambda})$ of $F_\lambda$ converges cosmically to $(\mu_f, \nu_f)$ as $\lambda\to0^+$, that is, for every continuous function $\xi:\R^n\to\R$ for which the limit
$$\overline{\xi}(\theta) = \lim_{\lambda\to\infty} \frac{\xi(\lambda \theta)}{\lambda}$$
exists (in the finite sense) uniformly in $\theta\in S^{n-1}$, we have
\begin{equation} \label{eq:universe}
\int_{\R^n} \xi(y) \, d\mu_f(y) + \int_{S^{n-1}}\overline{\xi}(\theta)\, d\nu_f(\theta) = \lim_{\lambda\to 0^+} \int_{\R^n} \xi(y) \, d\mu_{F_\lambda}(y) + \int_{S^{n-1}}\overline{\xi}(\theta)\, d\nu_{F_\lambda}(\theta).
\end{equation}
However, since $F_\lambda$ are $C^1$ and positive on $\R^n$, they are in particular essentially continuous. Thus $\nu_{F_\lambda}=0$ and $\mu_{F_\lambda}$ is the moment measure of $F_\lambda$. Plugging $\xi(y)=|y|$ in \eqref{eq:universe}, for which $\overline{\xi}(\theta)\equiv1$, we finally conclude that
$$\int_{\R^n} |y| \, d\mu_f(y) + \nu_f(S^{n-1}) = \lim_{\lambda\to0^+} \int_{\R^n} |y| \, d\mu_{F_\lambda}(y) \ls Cn$$
by Theorem \ref{th:moment-mu-f} as the functions $F_\lambda$ are isotropic log-concave.
\end{proof}

Combining Theorem~\ref{th:nu-f} with the generalized co-area formula \eqref{eq:gen-coarea} of \cite[Theorem~3.2]{Rotem-2022}, we derive the following general statement about the perimeter of super-level sets of isotropic log-concave functions.

\begin{corollary}\label{cor:H}
There exists a universal constant $C>0$ such that for every isotropic log-concave function $f$ on $\mathbb{R}^n$,
\begin{equation}
\int_0^\infty \mathcal{H}^{n-1}(\partial \{x: \ f(x)\gr t\}) \, dt \ls Cn. \end{equation}
\end{corollary}

\bigskip

\noindent \textbf{\S 6.3. Projections of a log-concave function.} Let $f:{\mathbb R}^n\to [0,\infty)$ be a log-concave function. 
Given $E\in G_{n,k}$, where $G_{n,k}$ is the Grassmann manifold of $k$-dimensional subspaces of ${\mathbb R}^n$, the ``section" of 
$f$ with $E$ is the restriction $f\big|_E$ of $f$ onto $E$ and the ``projection" or ``shadow" of $f$ onto $E$ is the function
$$P_Ef(x):=\max\{f(x+y):y\in E^{\perp}\},\quad x\in E$$
where $E^{\perp}$ is the orthogonal subspace of $E$. For every log-concave function $g$ and any $t>0$ we define $\tilde{R}_t(g)=\{x:g(x)\gr t\}$, $t>0$. Note that $R_s(g)=\tilde{R}_{e^{-s}g(0)}(g)$ for every $s>0$. It is not hard to check that
\begin{equation}\label{eq:RP=PR}\tilde{R}_t(P_Eg)=P_E(\tilde{R}_t(g))\quad\text{and hence}\quad R_s(P_Eg)=P_E(R_s(g))\end{equation}
for every $t,s>0$. We denote by $\nu_{n,k}$ the uniform probability measure on $G_{n,k}$.

\begin{theorem}\label{th:projections-f}Let $\mu$ be an isotropic log-concave probability measure on $\mathbb{R}^n$ with density $f=e^{-\psi}$
such that $f(0)=\|f\|_{\infty}$. Then,
$$\int_{G_{n,n-1}}\|P_Ef\|_1d\nu_{n,n-1}(E)\ls C\sqrt{n},$$
where $C>0$ is an absolute constant.
\end{theorem}

\begin{proof}Cauchy's surface area formula asserts that
\begin{equation}\label{eq:cauchy}S(K)= \frac{n\omega_n}{\omega_{n-1}}\int_{G_{n,n-1}}\vol_{n-1}(P_E(K))\,d\nu_{n,n-1}(E)\end{equation}
for every convex body $K$ in $\mathbb{R}^n$. On the other hand, for any $E\in G_{n,n-1}$ we have
\begin{align*}\|P_Ef\|_1 &=\int_0^{\|f\|_{\infty}}\vol_{n-1}(\tilde{R}_t(P_Ef))\,dt=\int_0^{\infty}e^{-s}\|f\|_{\infty}\vol_{n-1}(R_s(P_Ef))\,ds\\
&=\int_0^{\infty}e^{-s}\|f\|_{\infty}\vol_{n-1}(P_E(R_s(f)))\,ds,
\end{align*}
taking into account \eqref{eq:RP=PR}. Combining the above, we write
\begin{align*}\int_{G_{n,n-1}}\|P_Ef\|_1d\nu_{n,n-1}(E) &=\int_{G_{n,n-1}}\int_0^{\infty}e^{-s}\|f\|_{\infty}\vol_{n-1}(P_E(R_s(f)))\,ds\,d\nu_{n,n-1}(E)\\
&=\int_0^{\infty}e^{-s}\|f\|_{\infty}\int_{G_{n,n-1}}\vol_{n-1}(P_E(R_s(f)))\,d\nu_{n,n-1}(E)\,ds\\
&=\frac{\omega_{n-1}}{n\omega_n}\int_0^{\infty}e^{-s}\|f\|_{\infty}S(R_s(f))\,ds\\
&=\frac{\omega_{n-1}}{n\omega_n}\int_0^\infty \mathcal{H}^{n-1}(\partial\{x: \ f(x)\geq t\}) \, dt\ls C\sqrt{n},
\end{align*}
from Corollary~\ref{cor:H} and the fact that $\frac{\omega_{n-1}}{n\omega_n}\approx \frac{1}{\sqrt{n}}$.
\end{proof}

Note that if $K$ is an isotropic convex body in $\mathbb{R}^n$ and $f=L_K^n\mathds{1}_{K/L_K}$ then $P_Ef=L_K^n\mathds{1}_{P_E(K/L_K)}$. Therefore,
\begin{align*}\int_{G_{n,n-1}}\|P_Ef\|_1d\nu_{n,n-1}(E) &=L_K\int_{G_{n,n-1}}\vol_{n-1}(P_E(K))\,d\nu_{n,n-1}(E)\\
&\approx \int_{G_{n,n-1}}\vol_{n-1}(P_E(K))\,d\nu_{n,n-1}(E)
\end{align*}
in this case. Consider the example of the cube $Q_n=\left[-\tfrac{1}{2},\tfrac{1}{2}\right]^n$. Then,
\begin{align*}\int_{G_{n,n-1}}\vol_{n-1}(P_E(Q_n))\,d\nu_{n,n-1}(E) &=\int_{S^{n-1}}\vol_{n-1}(P_{\xi^{\perp}}(Q_n))\,d\sigma(\xi)\\
&=\frac{\omega_{n-1}}{n\omega_n}S(Q_n)=\frac{2\omega_{n-1}}{\omega_n}\approx\sqrt{n}.
\end{align*}
This shows that the upper bound that we obtained in Theorem~\ref{th:projections-f} is optimal with respect to the dimension.

\bigskip

\noindent {\bf Acknowledgements.} The research of the first named author was funded, in whole or in part, by the French National Research Agency (ANR) under grant number ANR-24-ERCS-0011. For the purpose of open access dissemination, the author has applied a CC-BY open access license to any Author Accepted Manuscript (AAM) resulting from this submission. The third named author acknowledges support by a PhD scholarship from the National Technical University of Athens.

\bigskip 


\footnotesize
\bibliographystyle{amsplain}

\begin{thebibliography}{100}
\footnotesize

\bibitem{AL25} \textrm{G.\ Aishwarya and D.\ Li}, \textit{Entropic and functional forms of the dimensional Brunn-Minkowski inequality in Gauss space}, Math. Ann.  393 (2025), no.~3-4, 3025--3042.
\bibitem{Alonso-Gonzalez-Jimenez-Villa-2018} \textrm{D.\ Alonso-Guti\'{e}rrez, B.\ Gonz\'{a}lez Merino, C.\ Jim\'{e}nez and R.\ Villa}, 
\textit{John's ellipsoid and the integral ratio of a log-concave function}, J. Geom. Anal. 28 (2018), no.~2, 1182--1201.
\bibitem{Alves-Gonzalez-Villa-2023} \textrm{L.\ Alves da Silva, B.\ Gonz\'{a}lez Merino and R.\ Villa}, \textit{Some remarks on Petty projection of log-concave functions}, 
J. Geom. Anal. 33 (2023), no.~8, Paper No. 260.
\bibitem{AGA-book} \textrm{S.\ Artstein-Avidan, A.\ Giannopoulos and V.\ D.\ Milman},
\textit{Asymptotic Geometric Analysis, Vol. I}, Mathematical Surveys and Monographs, 202.
American Mathematical Society, Providence, RI, 2015. xx+451 pp.
\bibitem{AGA-book-2} \textrm{S.\ Artstein-Avidan, A.\ Giannopoulos and V.\ D.\ Milman},
\textit{Asymptotic Geometric Analysis, Vol. II}, Mathematical Surveys and Monographs, 261.
American  Mathematical Society, Providence, RI, 2021. xxxvii+645 pp.
\bibitem{Ball-1988} \textrm{K. M. Ball}, \textit{Logarithmically concave functions and sections of convex sets in ${\mathbb R}^n$},
Studia Math. {88} (1988), 69--84.
\bibitem{Bal93} \textrm{K.~M. Ball}, \textit{The reverse isoperimetric problem for Gaussian measure}, Discrete Comput. Geom. 10 (1993), no.~4, 411--420.

\bibitem{Bizeul-2025} \textrm{P.\ Bizeul}, \textit{The slicing conjecture via small ball estimates}, Ann. Probab. (to appear).\\
({\tt https://arxiv.org/abs/2501.06854})

\bibitem{Borell-1974} \textrm{C.\ Borell}, \textit{Convex measures on locally convex spaces},
Ark. Mat. {12} (1974), 239--252.
\bibitem{BLYZ12} \textrm{K.\ B\"or\"oczky, E.\ Lutwak, D.\ Yang, G.\ Zhang}, \textit{The log-Brunn-Minkowski inequality}, Adv. Math. 231 (2012), no.~3-4, 1974--1997.
\bibitem{Bourgain-1986} \textrm{J.\ Bourgain}, \textit{On high-dimensional maximal functions associated to convex bodies},
Amer. J. Math. {108} (1986), 1467--1476.

\bibitem{BGHT} \textrm{S.\ Brazitikos, A.\ Giannopoulos, A.\ Hmadi and N.\ Tziotziou}, \textit{On the maximal perimeter of isotropic log-concave probability measures}, 
Preprint ({\tt https://arxiv.org/abs/2602.03831}).
\bibitem{BGVV-book} \textrm{S.\ Brazitikos, A.\ Giannopoulos, P.\ Valettas and B-H.\ Vritsiou}, \textit{Geometry of isotropic
convex bodies}, Mathematical Surveys and Monographs, 196. American Mathematical Society, Providence, RI, 2014. xx+594 pp.
\bibitem{Buser-1982} \textrm{P.\ Buser}, \textit{A note on the isoperimetric constant}, Ann. Sci. \'{E}cole Norm. Sup.
{15} (1982), 213--230.
\bibitem{Cattiaux-Guillin-2020} \textrm{P.\ Cattiaux and A.\ Guillin}, \textit{On the Poincar\'{e} constant of log-concave measures}, 
Geometric aspects of functional analysis. Vol. I, 171--217, Lecture Notes in Math., 2256, Springer, Cham, 2020.


\bibitem{Cheeger-1970} \textrm{J.\ Cheeger}, \textit{A lower bound for the smallest eigenvalue of the Laplacian},
Problems in analysis (Sympos. in honor of Salomon Bochner, Princeton Univ., Princeton, N.J., 1969), pp. 195--199, Princeton Univ. Press, Princeton, NJ, 1970.
\bibitem{Colesanti-Fragala-2013} \textrm{A.\ Colesanti and I.\ Fragal\`{a}}, \textit{The first variation of the total mass of log-concave functions and related inequalities},
Adv. Math. 244 (2013), 708--749.
\bibitem{CLM17} \textrm{A.\ Colesanti, G.~V.\ Livshyts and A.\ Marsiglietti}, \textit{On the stability of Brunn-Minkowski type inequalities}, J. Funct. Anal. 273 (2017), no.~3, 1120--1139.
\bibitem{Cordero-Eskenazis-2025} \textrm{D.\ Cordero-Erausquin and A.\ Eskenazis}, \textit{Concavity principles for weighted marginals}, Preprint.\\
({\tt https://arxiv.org/abs/2506.16941})
\bibitem{Cordero-Klartag-2015} \textrm{D.\ Cordero-Erausquin and B.\ Klartag}, \textit{Moment measures}, J. Funct. Anal. 268 (2015), no.~12, 3834--3866.
\bibitem{Cordero-Rotem-2023} \textrm{D.\ Cordero-Erausquin and L.\ Rotem}, \textit{Improved log-concavity for rotationally invariant measures of symmetric 
convex sets}, Ann. Probab. 51 (2023), no.~3, 987--1003. 
\bibitem{EK08} \textrm{R. Eldan and B. Klartag}, \textit{Pointwise estimates for marginals of convex bodies}, J. Funct. Anal.  254 (2008), no.~8, 2275--2293.
\bibitem{Eskenazis-Moschidis-2021} \textrm{A.~Eskenazis and G.~Moschidis}, \textit{The dimensional Brunn-Minkowski inequality in
Gauss space}, J. Funct. Anal. 280 (2021), Issue~6, Paper No.~108914.
\bibitem{Evans} \textrm{L.~C. Evans and R.~F. Gariepy}, \textit{Measure theory and fine properties of functions}, Studies in Advanced Mathematics, CRC, Boca Raton, FL, 1992.
\bibitem{FR26} \textrm{T. Falah and L. Rotem}, \textit{On the functional Minkowski problem}, Calc. Var. Partial Differential Equations {\bf 65} (2026), no.~3, Paper No.~77.
\bibitem{Fradelizi-1997} \textrm{M.\ Fradelizi}, \textit{Sections of convex bodies through their centroid},
Arch. Math. (Basel) {69} (1997), no. 6, 515--522.

\bibitem{Gardner-Zvavitch-2010} \textrm{R. J. Gardner and A. Zvavitch}, \textit{Gaussian Brunn-Minkowski inequalities},
Trans. Amer. Math. Soc. 362 (10) (2010) 5333--5353.

\bibitem{Giannopoulos-Tziotziou-2025} \textrm{A.\ Giannopoulos and N.\ Tziotziou}, \textit{Moments of the Cram\'{e}r
transform of log-concave probability measures}, J.\ Funct.\ Anal.\ 290 (2026), Issue~10, Paper No.~111423. 
\bibitem{Giannopoulos-Tziotziou-2026} \textrm{A.\ Giannopoulos and N.\ Tziotziou}, \textit{Regular functional covering numbers}, 
International Mathematics Research Notices, Volume 2026, Issue 5, Paper: rnag035.
\bibitem{Giannopoulos-Pafis-Tziotziou-2025} \textrm{A.\ Giannopoulos, M.\ Pafis and N.\ Tziotziou}, 
\textit{The isotropic constant in the theory of high-dimensional convex bodies}, Bull. Hellenic Math. Soc. 69 (2025), 89--188.
\bibitem{GK25} \textrm{V. Gorev and E.~D. Kosov}, \textit{Functional analogs of the Shephard, Busemann-Petty and Milman problems}, J. Anal. Math. 157 (2025), no.~2, 511--536.
\bibitem{Guan-preprint} \textrm{Q.~Y.~Guan}, \textit{A note on Bourgain's slicing problem}, Preprint ({\tt https://arxiv.org/abs/2412.09075}).
\bibitem{Klartag-2007} \textrm{B.\ Klartag}, \textit{A central limit theorem for convex sets}, Invent. Math. {168} (2007), 91--131.
\bibitem{Klartag-2023} \textrm{B.~Klartag}, \textit{Logarithmic bounds for isoperimetry and slices of convex sets}. Ars Inveniendi
Analytica, Paper No. 4, 17pp, 2023.
\bibitem{Klartag-Lehec-2025} \textrm{B.~Klartag and J.~Lehec}, \textit{Affirmative resolution of Bourgain's slicing problem using
Guan's bound}, Geom. Funct. Anal. (GAFA), vol.~35, (2025), 1147--1168.
\bibitem{KL25} \textrm{B.~Klartag and J.~Lehec}, \textit{Thin-shell bounds via parallel coupling}, Preprint.\\ 
({\tt https://arxiv.org/abs/2507.15495})



\bibitem{Kolesnikov-Livshyts-2021} \textrm{A.\ V.\ Kolesnikov and G.\ V.\ Livshyts}, \textit{On the Gardner-Zvavitch conjecture: symmetry in inequalities of Brunn-Minkowski type}, 
Adv. Math. 384 (2021), Paper No. 107689, 23 pp.

\bibitem{Kolesnikov-Milman-2018} \textrm{A.\ V.\ Kolesnikov and E.\ Milman}, \textit{Poincar\'{e} and Brunn-Minkowski inequalities on the boundary of weighted Riemannian manifolds}, 
Amer. J. Math. 140 (2018), no. 5, 1147--1185. 
\bibitem{Kolesnikov-Milman-2022} \textrm{A.\ V.\ Kolesnikov and E.\ Milman}, \textit{Local $L^p$-Brunn--Minkowski inequalities for $p<1$}, 
Mem. Amer. Math. Soc. 277 (2022), no.~1360, v+78 pp. 

\bibitem{Ledoux-1994} \textrm{M.\ Ledoux}, \textit{A simple analytic proof of an inequality by P. Buser},
Proc. Am. Math. Soc. {121} (1994), 951--959.
\bibitem{Liv13} \textrm{G.~V. Livshyts}, \textit{Maximal surface area of a convex set in $\mathbb{R}^n$ with respect to exponential rotation invariant measures}, J. Math. Anal. Appl. {404} (2013), no.~2, 231--238.
\bibitem{Liv14} \textrm{G.~V. Livshyts}, \textit{Maximal surface area of a convex set in $\mathbb{R}^n$ with respect to log concave rotation invariant measures}, in  Geometric aspects of functional analysis, 355--383, Lecture Notes in Math., 2116, Springer, Cham.
\bibitem{Liv15} \textrm{G.~V. Livshyts}, \textit{Maximal surface area of polytopes with respect to log-concave rotation invariant measures}, Adv. in Appl. Math. 70 (2015), 54--69.
\bibitem{Liv21} \textrm{G.~V. Livshyts}, \textit{Some remarks about the maximal perimeter of convex sets with respect to probability measures}, Commun. Contemp. Math.  23 (2021), no.~5, Paper No. 2050037, 19 pp.
\bibitem{Livshyts-2023} \textrm{G.\ V.\ Livshyts}, \textit{A universal bound in the dimensional Brunn-Minkowski inequality for log-concave measures}, 
Trans. Amer. Math. Soc. 376 (2023), no. 9, 6663--6680. 
\bibitem{Liv24} \textrm{G.~V. Livshyts}, \textit{On a conjectural symmetric version of Ehrhard's inequality}, Trans. Amer. Math. Soc.  377 (2024), no.~7, 5027--5085.
\bibitem{Livshyts-Marsiglietti-Nayar-Zvavitch-2017} \textrm{G.\ Livshyts, A.\ Marsiglietti, P.\ Nayar and A.\ Zvavitch},
\textit{On the Brunn--Minkowski inequality for general measures with applications to new isoperimetric-type inequalities}, 
Trans. Amer. Math. Soc. 369 (2017), no.~12, 8725--8742.
\bibitem{MMRR25} \textrm{A.\ Malliaris, J.\ Melbourne, C.\ Roberto and M.\ Roysdon}, \textit{Functional liftings of restricted geometric inequalities}, Preprint ({\tt https://arxiv.org/abs/2508.15247}).
\bibitem{Mazya-1960} \textrm{V. G. Maz'ya}, \textit{Classes of domains and imbedding theorems for function spaces},
Dokl. Acad. Nauk SSSR (Engl. transl. Soviet Math. Dokl., 1 (1961) 882--885) {3} (1960), 527--530.
\bibitem{Mazya-1962} \textrm{V. G. Maz'ya}, \textit{The negative spectrum of the higher-dimensional Schr\"{o}dinger
 operator}, Dokl. Akad. Nauk SSSR {144} (1962), 721--722.
\bibitem{EMilman-2009} \textrm{E.\ Milman}, \textit{On the role of convexity in isoperimetry, spectral-gap and concentration}, Invent. Math. 177 (2009), 1--43.
\bibitem{Naz03} \textrm{F.~L. Nazarov}, \textit{On the maximal perimeter of a convex set in $\mathbb{R}^n$ with respect to a Gaussian measure}, in Geometric aspects of functional analysis, 169--187, Lecture Notes in Math., 1807, Springer, Berlin.
\bibitem{RW98} \textrm{R.~T. Rockafellar and R.~J.-B. Wets}, \textit{Variational analysis},  Grundlehren der mathematischen Wissenschaften, 317. Springer-Verlag, Berlin, 1998. xiv+733 pp. 
\bibitem{Rotem-2022} \textrm{L.\ Rotem}, \textit{Surface area measures of log-concave functions}, J. Anal. Math. 147 (2022), no.~1, 373--400.
\bibitem{Rotem-2023} \textrm{L.\ Rotem}, \textit{The anisotropic total variation and surface area measures}, Geometric aspects of functional analysis, 297--312, 
Lecture Notes in Math., 2327, Springer, Cham, 2023.
\bibitem{Schneider-book} \textrm{R.\ Schneider}, \textit{Convex Bodies: The Brunn-Minkowski Theory},
Second expanded edition. Encyclopedia of Mathematics and its Applications, 151.
Cambridge University Press, Cambridge, 2014. xxii+736 pp.
\end{thebibliography}

\bigskip

\thanks{\noindent {\bf Keywords:} Dimensional Brunn–Minkowski, log-concave measures, isotropic position, convex bodies, functional inequalities, 
moment measures, gradient estimates.}

\smallskip

\thanks{\noindent {\bf 2020 MSC:} Primary 52A23; Secondary 52A38, 52A40, 60E15, 26D15, 46B06.}

\bigskip

\bigskip 

\medskip 

\noindent \textsc{Alexandros \ Eskenazis}: Institut de Math\'{e}matiques de Jussieu, Sorbonne Universit\'{e}, Paris, 75252, France.

\smallskip

\noindent \textit{E-mail:} \texttt{alexandros.eskenazis@imj-prg.fr}

\bigskip

\noindent \textsc{Apostolos \ Giannopoulos}: School of Applied Mathematical and Physical Sciences, National Technical University of Athens, Department of Mathematics, Zografou Campus, GR-157 80, Athens, Greece.

\smallskip

\noindent \textit{E-mail:} \texttt{apgiannop@math.ntua.gr}

\bigskip

\noindent \textsc{Natalia \ Tziotziou}: School of Applied Mathematical and Physical Sciences, National Technical University of Athens, Department of Mathematics, Zografou Campus, GR-157 80, Athens, Greece.

\smallskip

\noindent \textit{E-mail:} \texttt{tziotziounatalia@mail.ntua.gr}

\end{document}